\documentclass{kybernetika}

\usepackage{microtype}
\usepackage{authblk}
\usepackage[utf8]{inputenc}
\usepackage{xcolor}

\usepackage[sort]{cite}

\usepackage[frozencache,cachedir=minted-cache]{minted}
\colorlet{CodeBackground}{black!5!white}
\setminted{encoding=utf8, mathescape, bgcolor=CodeBackground, fontsize=\footnotesize}

\usepackage{amsmath, amssymb, amsthm}
\usepackage{graphicx, url}
\usepackage{appendix}
\usepackage{csquotes}
\usepackage{hyperref}
\usepackage[sort]{cleveref}
\usepackage{enumitem}

\newlist{paraenum}{enumerate}{1}
\setlist[paraenum]{wide, label=(\arabic*)}
\newlist{inparaenum}{enumerate*}{1}
\setlist[inparaenum]{label=(\arabic*)}

\usepackage{dsfont}

\usepackage{sansmath}

\usepackage{caption}
\usepackage{subcaption}
\usepackage{wrapfig}

\usepackage{tikz}
\usetikzlibrary{arrows}
\usetikzlibrary{positioning}

\usetikzlibrary{cd}

\usepackage{mathtools}
\newcommand{\defas}{\coloneqq}

\newcommand{\CIperp}{\mathrel{\text{$\perp\mkern-10mu\perp$}}}
\usepackage{ifthen}
\usepackage{listofitems}
\newcommand{\CI}[2][\CIperp]{{%
  \setsepchar{{:}/{|}/{,}}
  \ignoreemptyitems
  \readlist*\mylist{#2}
  \ifthenelse{\listlen\mylist[] = 2}{\mylist[2]}{}%
  [\mylist[1,1,1] #1 \ifthenelse{\listlen\mylist[1,1] = 2}{\mylist[1,1,2]}{\mylist[1,1,1]}%
  \ifthenelse{\listlen\mylist[1] = 2}{{} \mid \mylist[1,2]}{}]%
}}

\usepackage{xstring}

\newcommand{\dperp}{\perp_d}
\newcommand{\dsepgiven}[3]{[#1 \dperp #2 | #3]}

\newcommand{\Cperp}{\perp_{C^*}}

\newcommand{\Csepgiven}[3]{\CI[\Cperp]{#1,#2|#3}}
\newcommand{\notCsepgiven}[3]{\CI[\not\Cperp]{#1,#2|#3}}
\newcommand{\globalCsep}[2]{\mathcal{M}_\ast(#1, #2)}

\usepackage{etoolbox}
\NewDocumentCommand{\PolyCI}{e{_}m}{S_{#1}\CI{#2}}

\DeclareMathOperator{\cone}{cone}

\newcommand{\FN}[1]{\textcolor{orange}{FN: #1}}

\DeclareMathOperator{\pa}{pa}

\newcommand{\indep}{\mathrel{\text{$\perp\mkern-10mu\perp$}}}

\newcommand{\ot}{\leftarrow}

\newcommand{\cg}{\mathcal{G}}
\newcommand{\tr}{\mathrm{tr}}
\newcommand{\cgtr}{\mathcal{G}^{\tr}_C}
\newcommand{\critdag}[1]{\mathcal{G}^{\ast}_{#1}}

\newcommand{\cf}{\mathcal{F}}
\newcommand{\picrit}[3][\pi]{{#1}^{#2 #3}_{\text{crit}}}

\newcommand{\initial}[1][C]{\mathrm{in}_{#1}}

\DeclareMathOperator{\Newt}{Newt}

\numberwithin{equation}{section}
\theoremstyle{plain}

\newtheorem{theorem}{Theorem}[section]

\newtheorem{conjecture}[theorem]{Conjecture}
\newtheorem{corollary}[theorem]{Corollary}
\newtheorem{lemma}[theorem]{Lemma}
\newtheorem{proposition}[theorem]{Proposition}

\newtheorem*{theorem*}{Theorem}
\newtheorem*{problem*}{Problem}

\theoremstyle{definition}
\newtheorem{remark}[theorem]{Remark}
\newtheorem{definition}[theorem]{Definition}

\newtheorem{example}[theorem]{Example}

\begin{document}

\pagestyle{myheadings}
\title{Polyhedral aspects of maxoids}

\author{Tobias Boege, Kamillo Ferry, Benjamin Hollering, Francesco Nowell}

\contact{Tobias}{Boege}%
{Department of Mathematics and Statistics,
UiT -- The Arctic University of Norway,
N-9037 Tromsø, Norway}%
{post@taboege.de}

\contact{Kamillo}{Ferry}%
{Institut für Mathematik,
Technische Universität Berlin,
Straße des 17.~Juni~136,
D-10623 Berlin, Germany}%
{ferry@math.tu-berlin.de}

\contact{Benjamin}{Hollering}%
{Max Planck Institute for Mathematics in the Sciences,
Inselstraße~22,
D-04103 Leipzig, Germany}%
{benjamin.hollering@mis.mpg.de}

\contact{Francesco}{Nowell}%
{Institut für Mathematik,
Technische Universität Berlin,
Straße des 17.~Juni~136,
D-10623 Berlin, Germany}%
{nowell@math.tu-berlin.de}

\markboth{T.~Boege, K.~Ferry, B.~Hollering, and F.~Nowell}{Polyhedral aspects of maxoids}

\maketitle

\begin{abstract}
The conditional independence (CI) relation of a distribution in a max-linear Bayesian network depends on its weight matrix through the $C^\ast$-separation criterion. These CI~models, which we call \emph{maxoids}, are compositional graphoids which are in general not representable by Gaussian random variables. We prove that every maxoid can be obtained from a transitively closed weighted DAG and show that the stratification of generic weight matrices by their maxoids yields a polyhedral~fan. We also use this connection to polyhedral geometry to develop an algorithm for solving the conditional independence implication problem for maxoids.
\end{abstract}


\section{Introduction}
Linear structural equation models, also known as Bayesian networks, are of critical importance in modern data science and statistics through their applications to  causality~\cite{pearl2009causality} and probabilistic inference~\cite{koller2009probabilistic}. These statistical models use directed acyclic graphs (DAGs) to represent causal relationships and conditional independencies between random variables. Recently, there has been a focus on developing graphical models which are able to capture causal relations between extreme events. The~two main approaches employ \emph{H\"ussler--Reiss distributions} \cite{engelke2024graphicalmodelsmultivariateextremes, engelke2025extremesstructuralcausalmodels} and max-linear Bayesian networks, the latter of which are the main subject of this paper.

\enlargethispage{1.5em}
\emph{Max-linear Bayesian networks} (MLBNs), were introduced in \cite{gissibl2018max} to model \emph{cascading failures}. They are used in areas where these failures lead to catastrophic events, such as financial risk and water contamination \cite{leigh2019framework, rochet1996interbank}. A random vector $X = (X_1, \ldots, X_n)$ is distributed according to the max-linear model on a DAG $\cg$ if it satisfies the system of \emph{recursive structural equations}
\begin{equation}
\label{eqn:mlbn}
    X_j =\bigvee_{i \in \pa(j)}\mathrm{e}^{c_{ij}}X_i \vee Z_j, ~ \quad ~  Z_j \geq 0\text{ and } c_{ij}\in\mathbb{R},
\end{equation}
where $\vee = \max$, the exponentials $\mathrm{e}^{c_{ij}}$ are positive edge weights, $\pa(j)$ is the set of parents of $j$ in $\cg$, and the $Z_j$ are independent, atom-free, continuous random variables.

The structural equations mimic Bayesian networks in the extreme value setting. Despite this similarity, the conditional independence (CI) theory of MLBNs turns out to be more subtle in certain aspects than that of classical Bayesian networks which are governed by the well-known d-separation criterion. In addition to the d-separations of the DAG, a max-linear model may satisfy other CI~statements which depend on the weight matrix~$C$ whose entries $c_{ij}$ appear in~\eqref{eqn:mlbn}. Am\'{e}ndola et~al.~\cite{MaxLinearCI} observed that multiple distinct CI~structures can arise for the same DAG, each for a set of $C$-matrices with positive Lebesgue measure. They introduced the graphical $\ast$-separation criterion which is complete but not strongly complete for CI~implication in MLBNs, and the $C^\ast$-separation criterion which takes $C$ into account and completely characterizes the CI~structure of an MLBN. Moreover, the following chain of implications from d- over $\ast$- to $C^\ast$-separation is valid for all MLBNs:
\[
\CI[\perp_d]{i,j|L} \implies \CI[\perp_\ast]{i,j|L} \implies \CI[\Cperp]{i,j|L}.
\]
In this paper, we focus on the CI~structures which arise from $C^\ast$-separation since they are the most refined according to these implications and MLBNs are generically faithful to them.

Max-linear Bayesian networks come with a rich polyhedral structure which becomes apparent after applying a logarithmic transformation to study \(\log{X}\).
This has been done before by Tran~\cite{Tran:2022} and Am\'{e}ndola and Ferry~\cite{Amendola.Ferry:2026}
to obtain results about MLBNs using methods from polyhedral and tropical geometry.
Our point of departure is the logarithmic version of the structural equations in \eqref{eqn:mlbn}:
\begin{equation}
\label{eqn:log-mlbn}
    \log X_j =\bigvee_{i \in \pa(j)} (c_{ij}+ \log X_i) \vee \log Z_j.
\end{equation}
The CI structure of an MLBN is entirely determined by the coefficient matrix~$C$. We call $\globalCsep{\cg}{C} \defas \{\, \CI{I,J|L} : \CI[\Cperp]{I,J|L}  \text{ in } (\cg,C) \,\}$ the \emph{maxoid} associated to the DAG $\cg$ with coefficient matrix~$C$ and note that this is essentially the \emph{global Markov property} of the weighted graph $(\cg, C)$. The mechanism of $C^\ast$-separation by which the maxoid can be read off from the weighted graph is explained in detail in \Cref{sec:poly-geo-star-sep}.

The first main result of this paper endows maxoids with a geometric structure. Namely, the set of distinct maxoids associated to a fixed DAG $\cg$ are in correspondence with the cones of a complete fan for which we provide an explicit representation via inequalities.

\begin{theorem*}
For any DAG $\cg$ there is a hyperplane arrangement $\mathcal{H}_\cg \subseteq \mathbb{R}^E$ such that for every $C \in \mathbb{R}^E \setminus \mathcal{H}_\cg$ the set
\begin{align*}
    \cone_\cg(C) := \big\{ C' \in \mathbb{R}^E \setminus \mathcal{H}_\cg : \globalCsep{\cg}{C} = \globalCsep{\cg}{C'}\big\}
\end{align*}
is a full-dimensional open polyhedral cone. The collection of all closures of such cones for a fixed $\cg$ forms a complete polyhedral fan $\cf_\cg$ in $\mathbb{R}^E$. Moreover the map which sends a cone of $\cf_\cg$ to its maxoid is an inclusion-reversing surjection.
\end{theorem*}

One immediate consequence of the above theorem together with the results of \cite{MaxLinearCI} is that the maximal cones of $\cf_\cg$ correspond to the distinct CI structures which can arise from a max-linear Bayesian network with positive Lebesgue measure for the choice of~$C$. In~this sense, the maximal cones correspond to the \emph{generic} CI structures of an MLBN supported on~$\cg$. Similarly, we call a weight matrix $C$ \emph{generic} if it does not lie on~$\mathcal{H}_\cg$. As~we will show in \Cref{sec:poly-geo-star-sep}, if there exist two nodes $i, j \in V(\cg)$ such that there are at least two distinct paths between $i$ and $j$, then $\cf_\cg$ has at least two distinct full-dimensional cones. This provides a strong contrast to classical linear structural equation models which are generically faithful to d-separation, and thus almost every distribution in the model exhibits the same CI~structure; cf.~\cite{lauritzen-graphical-models}.
Our results also elucidate which CI structures may arise from a given graph and when two graphs exhibit the same generic CI structure. This is critical for determining if the graph structure may be recovered using only conditional independencies as is typically done in constraint-based causal discovery algorithms, e.g.,\cite{spirtes2000causation}.

The polyhedral fan $\cf_\cg$ also provides a practical approach to the implication problem for conditional independence in MLBNs; see Studený's book \cite{Studeny} for an introduction to the implication problem. The set of weight matrices whose maxoids contain a given CI~statement is a union of polyhedral cones from~$\cf_\cg$. A set of this form is called a \emph{polyhedral set}. %
To decide a CI~implication for maxoids means checking whether one polyhedral set is contained in another. We show how modern \emph{satisfiability modulo theories (SMT)} solvers, which are developed primarily for formal verification purposes, can be applied to this formulation. They solve a series of linear programming problems and either conclusively determine that the implication is true or produce an explicit counterexample in the form of a weight matrix~$C$. Furthermore, we use the combinatorics of $C^\ast$-separation to establish general closure properties which every maxoid~satisfies.

\begin{theorem*}
Maxoids are compositional graphoids. Moreover, they satisfy the ``amalgamation property for blocking sets'':
\begin{align*}
  \CI{i,j|KM} \land \CI{i,j|LM} &\;\implies\; \CI{i,j|KLM},
\end{align*}
as well as the following strengthening of the Spohn property from \cite{SpohnProperties}:
\begin{align*}
  \CI{i,j|klM} \land \CI{k,l|iM} \land \CI{k,l|jM} &\;\implies\; \CI{k,l|M}, \\
  \CI{i,j|klM} \land \CI{k,l|iM} \land \CI{k,l|M} &\;\implies\; \CI{k,l|jM}.
\end{align*}
\end{theorem*}

The remainder of this paper is organized as follows. In \Cref{sec:poly-geo-star-sep} we recall the details of the $C^\ast$-separation criterion and use this to provide an explicit description of the linear inequalities which define $\cone_\cg(C)$ via the \emph{critical paths} of $\cg$. We then relate the set of maxoids arising from a DAG $\cg$ to the cones of the associated polyhedral~fan. In~\Cref{sec:maxoid-polytope} we introduce a polytope associated to $\cg$ which can be used to enumerate the generic maxoids, and discuss the similarities between this construction and that of the Gröbner fan of an ideal. \Cref{sec:ci-axioms} deals with the problem of CI~implication. We first present an SMT-based algorithm for solving implication problems which we implemented in the \verb|julia| programming language. Then we prove general closure properties of maxoids. Finally, we observe that maxoids are in general not representable by regular Gaussian distributions. These observations set maxoids apart from d-separation.

This paper is accompanied by snippets of code collected in a \verb|julia| package called \verb|Maxoids.jl| \cite{Maxoids.jl}.
It implements algorithms for generating and manipulating the conditional independence structures of MLBNs and
the related polyhedral objects we define. It builds upon the functionality for graphs, polyhedral geometry,
and conditional independence present in the computer algebra system \verb|OSCAR| \cite{OSCAR,OSCAR-book} as well as the \verb|Satisfiability.jl| package \cite{Satisfiability.jl} and the \verb|Z3| SMT solver \cite{Z3}.

\pagebreak
\noindent%
Thus, in order to run the code it is necessary to install both packages by running
\begin{minted}{julia-repl}
julia> using Pkg
julia> Pkg.add(["Oscar", "Maxoids"])
\end{minted}
and load them at the start of every session by calling
\begin{minted}{julia-repl}
julia> using Oscar, Maxoids
\end{minted}

The code contained in the \verb|julia| package and any additional data is provided on Zenodo under the record
 \url{https://zenodo.org/records/17737553}.

\section{The Polyhedral Geometry of \texorpdfstring{\(C^\ast\)}{C-star}-separation}
\label{sec:poly-geo-star-sep}

Let $\cg = (V, E)$ be a DAG on $\lvert V\rvert=n$ vertices and denote the set of coefficient matrices supported on $\cg$ by~$\mathbb{R}^E$, i.e., all $n \times n$ matrices $C$ with $c_{ij} = - \infty$ if $i \to j \notin E$ and $c_{ij}\in \mathbb{R}$ otherwise. We recall that a random vector $X$ is distributed according to the max-linear model on $\cg$ if it satisfies \cref{eqn:log-mlbn}. By \cite[Theorem 2.2]{gissibl2018max}, this is equivalent to $X$ satisfying the equation $ \log X = (C^\ast)^T \log Z$ where the logarithm is applied entrywise and the matrix-vector product is done in $\max$-plus arithmetic.
The matrix $C^\ast$ is the \emph{Kleene star} of $C$ whose entries are given by

\begin{align} \label{eqn:kleenestarentry}
(C^\ast)_{ij} = \max_{\pi \in P(i,j)} \sum_{e \in \pi} c_e \ ,
\end{align}
where $P(i, j)$ denotes the set of all directed paths from $i$ to $j$ in $\cg$. For a path $\pi \in P(i,j)$, we refer to the quantity $\sum_{e \in \pi} c_e$ as the \emph{weight} of $\pi$ and denote it by $\omega_C(\pi).$

A path $\pi'$ is \emph{critical} if $\omega_C(\pi') = \max_{\pi \in P(i, j)} \omega_C(\pi)$.
If there is a unique critical path between every pair of nodes $i, j$ then we say that $C$ is \emph{generic} and denote the unique path from $i$ to $j$ by $\picrit{i}{j}(\cg, C)$, omitting the pair $(\cg, C)$ when it is clear from context. Note that if $C$ is not generic, then its entries satisfy some non-trivial linear equation of the form $\omega_C(\pi) = \omega_C(\pi')$ for distinct $\pi, \pi' \in P(i,j)$. Hence, the set of generic matrices is the complement of a hyperplane arrangement $\mathcal{H}_\cg$ whose defining equations depend on~$\cg$. We~are now ready to introduce $C^\ast$-separation.

\begin{figure}
\begin{subfigure}[t]{0.10\linewidth}
\centering
\begin{tikzpicture}[inner sep=1pt]
\node[draw, circle] (i) at (0,.7) {\strut$i$};
\node[draw, circle] (j) at (0,-.7) {\strut$j$};
\draw[->] (i) -- (j);
\end{tikzpicture}
\caption{}
\label{fig:*-sep:a}
\end{subfigure}
\hfill
\begin{subfigure}[t]{0.15\linewidth}
\centering
\begin{tikzpicture}[inner sep=1pt]
\node[draw, circle] (i) at (-.5,-.7) {\strut$i$};
\node[draw, circle] (p) at (  0,.7) {\strut$p$};
\node[draw, circle] (j) at (+.5,-.7) {\strut$j$};
\draw[->] (p) -- (i);
\draw[->] (p) -- (j);
\end{tikzpicture}
\caption{}
\label{fig:*-sep:b}
\end{subfigure}
\hfill
\begin{subfigure}[t]{0.15\linewidth}
\centering
\begin{tikzpicture}[inner sep=1pt]
\node[draw, circle] (i) at (-.5,.7) {\strut$i$};
\node[draw, circle, fill=orange!40!white] (l) at (  0,-.7) {\strut$\ell$};
\node[draw, circle] (j) at (+.5,.7) {\strut$j$};
\draw[->] (i) -- (l);
\draw[->] (j) -- (l);
\end{tikzpicture}
\caption{}
\label{fig:*-sep:c}
\end{subfigure}
\hfill
\begin{subfigure}[t]{0.2\linewidth}
\centering
\begin{tikzpicture}[inner sep=1pt]
\node[draw, circle] (i) at (-.5,-.7) {\strut$i$};
\node[draw, circle] (p) at (  0,+.7) {\strut$p$};
\node[draw, circle, fill=orange!40!white] (l) at (+.5,-.7) {\strut$\ell$};
\node[draw, circle] (j) at (1.0,+.7) {\strut$j$};
\draw[->] (p) -- (i);
\draw[->] (p) -- (l);
\draw[->] (j) -- (l);
\end{tikzpicture}
\caption{}
\label{fig:*-sep:d}
\end{subfigure}
\hfill
\begin{subfigure}[t]{0.2\linewidth}
\centering
\begin{tikzpicture}[inner sep=1pt]
\node[draw, circle] (i) at (-.5,-.7) {\strut$i$};
\node[draw, circle] (p) at (  0,+.7) {\strut$p$};
\node[draw, circle, fill=orange!40!white] (l) at ( .5,-.7) {\strut$\ell$};
\node[draw, circle] (q) at (1.0,+.7) {\strut$q$};
\node[draw, circle] (j) at (1.5,-.7) {\strut$j$};
\draw[->] (p) -- (i);
\draw[->] (p) -- (l);
\draw[->] (q) -- (l);
\draw[->] (q) -- (j);
\end{tikzpicture}
\caption{}
\label{fig:*-sep:e}
\end{subfigure}
\caption{%
The types of $\ast$-connecting paths between $i$ and $j$ given $L$ in a critical DAG~$\critdag{C,L}$. The~colored colliders $\ell$ must belong to $L$; the non-colliders $p,q$ must not belong to~$L$.}
\label{fig:*-sep}
\end{figure}

\begin{definition} \label{def:*-sep}
Let $(\cg, C)$ be a weighted DAG with vertex set~$V$ and \(L \subseteq V\). The \emph{critical DAG} $\critdag{C,L}$ is the DAG on $V$ such that $i \to j \in E(\critdag{C,L})$ whenever $i$ and $j$ are connected via a directed path, and no critical path from $i$ to $j$ in $\cg$ intersects~$L$.

Two nodes $i$ and $j$ are \emph{$C^\ast$-connected given $L$} if there exists a path from $i$ to $j$ in $\critdag{C, L}$ of the form pictured in \Cref{fig:*-sep}. If no such path exists, then $i$ and $j$ are \emph{$C^\ast$-separated given $L$} which is denoted $\Csepgiven{i}{j}{L}$.
\end{definition}

\begin{theorem}[{\cite[Theorem~6.18]{MaxLinearCI}}]
Let $(\cg, C)$ be a weighted DAG and $X$ be a random vector distributed according to the max-linear model on $(\cg, C)$. Then
\[
  \Csepgiven{i}{j}{L} \implies \CI{i, j | L}.
\]
Moreover, the converse holds for all but a Lebesgue null set of weight~matrices~$C$.
\end{theorem}

$C^\ast$-separation generally entails more CI~statements than d-separation. In particular note that a $\ast$-connecting path can have at most one collider in its conditioning set whereas d-separation allows any number of colliders in a connecting path. Moreover, it suffices to block only a single critical path from $i$ and $j$ in order to separate them. %

\begin{example} \label{eg:diamond}
Consider the diamond graph $\cg$ with weight matrix $C$: \\[.5em]
\noindent\begin{minipage}{\textwidth}
\centering
\begin{minipage}[c]{\dimexpr0.3\textwidth}
\begin{tikzpicture}[inner sep=1pt, scale=0.9]
\node[draw, circle] (1) at (0,1) {\strut$1$};
\node[draw, circle] (2) at (-1,0) {\strut$2$};
\node[draw, circle] (3) at (1,0) {\strut$3$};
\node[draw, circle] (4) at (0,-1) {\strut$4$};
\draw[->] (1) -- (2);
\draw[->] (1) -- (3);
\draw[->] (2) -- (4);
\draw[->] (3) -- (4);
\end{tikzpicture}
\end{minipage}
\begin{minipage}[c]{\dimexpr0.3\textwidth}
\vfil
\[
  C = \begin{pmatrix}
  -\infty & c_{12} & c_{13} & -\infty \\
  -\infty & -\infty & -\infty & c_{24} \\
  -\infty & -\infty & -\infty & c_{34}\\
  -\infty & -\infty & -\infty & -\infty
  \end{pmatrix}
\]
\vfil
\end{minipage}
\end{minipage} \\[.5em]
\iffalse
Observe that $P(1, 4) = \{ \pi_2, \pi_3\}$ where $\pi_i = 1 \to i \to 4$. If $C$ satisfies $\omega_C(\pi_2) > \omega_C(\pi_3)$, then $\critdag{C, \{ 2\}}$ is exactly the diamond above because $\pi_2$ is a critical path from $1$ to $4$ which intersects the conditioning set $\{2\}$. Since neither $\pi_2$ nor $\pi_3$ are of the forms displayed in \Cref{fig:*-sep}, this MLBN satisfies~$\Csepgiven{1}{4}{2}$. On the other hand, if $\omega_C(\pi_3) > \omega_C(\pi_2)$ then a similar argument yields that $\Csepgiven{1}{4}{3}$. Thus we get two distinct maxoids which correspond to whether $\pi_2$ or $\pi_3$ is the critical path. Moreover, the maxoid $\globalCsep{\cg}{C}$ is completely determined by which side of the hyperplane
\[
c_{12} + c_{24} = \omega_C(\pi_2) = \omega_C(\pi_3) = c_{13} + c_{34}
\]
the matrix $C$ lies on.
\fi
Observe that $P(1, 4) = \{ \pi_2, \pi_3\}$ where $\pi_i = 1 \to i \to 4$. As $\dsepgiven{1}{4}{\{2,3\}}$ holds, the corresponding CI statement $\CI{1,4|2,3}$ holds in an MLBN on $\cg$ for any coefficient matrix $C$.
We now show that the $C^\ast$-criterion gives rise to additional conditional independence, which depends on the critical path structure of $(\cg,C)$.

If $C$ satisfies $\omega_C(\pi_2) > \omega_C(\pi_3)$, then $\critdag{C, \{ 2\}}$ is exactly the diamond above because $\pi_2$ is a critical path from $1$ to $4$ which intersects the conditioning set $\{2\}$. Since neither $\pi_2$ nor $\pi_3$ are of the forms displayed in \Cref{fig:*-sep}, this MLBN satisfies~$\Csepgiven{1}{4}{2}$. Conversely, the graph $\critdag{C, \{ 3\}}$ contains the edge $1 \rightarrow 4$ due to the fact that the $\pi_2 = 1 \rightarrow 2 \rightarrow 4$ is not blocked by $\{3\}$. The edge $1 \rightarrow 4$ is a path of type (a) in $\critdag{C, \{ 3\}}$ so by \Cref{def:*-sep}, $\notCsepgiven{1}{4}{3}$.

\begin{center}
\begin{minipage}[c]{\dimexpr0.4\textwidth}
\begin{center}
  $\critdag{C,\{2\}}$
\end{center}
\begin{center}
    \begin{tikzpicture}[inner sep=1pt, scale=0.9]
\node[draw, circle] (1) at (0,1) {\strut$1$};
\node[draw, circle, fill=orange!40!white] (2) at (-1,0) {\strut$2$};
\node[draw, circle] (3) at (1,0) {\strut$3$};
\node[draw, circle] (4) at (0,-1) {\strut$4$};
\draw[->] (1) -- (2);
\draw[->] (1) -- (3);
\draw[->] (2) -- (4);
\draw[->] (3) -- (4);
\end{tikzpicture}

\end{center}
\begin{center}
$\Csepgiven{1}{4}{2}$
\end{center}
\end{minipage}
\begin{minipage}[c]{\dimexpr0.4\textwidth}

\begin{center}
  $\critdag{C,\{3\}}$
\end{center}
\begin{center}
    \begin{tikzpicture}[inner sep=1pt, scale=0.9]
\node[draw, circle] (1) at (0,1) {\strut$1$};
\node[draw, circle] (2) at (-1,0) {\strut$2$};
\node[draw, circle, fill=orange!40!white] (3) at (1,0) {\strut$3$};
\node[draw, circle] (4) at (0,-1) {\strut$4$};
\draw[->] (1) -- (2);
\draw[->] (1) -- (3);
\draw[->] (2) -- (4);
\draw[->] (3) -- (4);
\draw[->] (1) -- (4);
\end{tikzpicture}
\end{center}
\begin{center}
$\notCsepgiven{1}{4}{3}$
\end{center}
\end{minipage}
\end{center}

On the other hand, if $C$ is such that $\omega_C(\pi_3) > \omega_C(\pi_2)$ then a similar argument yields that $\Csepgiven{1}{4}{3}$ and $\notCsepgiven{1}{4}{2}$. Thus we get two distinct maxoids which correspond to whether $\pi_2$ or $\pi_3$ is the critical path. Moreover, the maxoid $\globalCsep{\cg}{C}$ is completely determined by which side of the hyperplane
\[
c_{12} + c_{24} = \omega_C(\pi_2) = \omega_C(\pi_3) = c_{13} + c_{34}
\]
the matrix $C$ lies on.
\end{example}

Our goal in the remainder of this section is to develop the observations from the previous example into a general result which connects weighted DAGs and their maxoids using polyhedral geometry. We begin with a sequence of lemmas which further elucidate the connection between the critical paths in $(\cg, C)$ and the CI~structure $\globalCsep{G}{C}$.

\begin{lemma} \label{lem:markovcritpaths}
    Let $(\cg, C)$ and $(\cg', C')$ be two weighted DAGs on the same node set and generic weights $C$ and $C'$. Then
    \[
        \globalCsep{\cg}{C} = \globalCsep{\cg'}{C'} \iff \picrit{i}{j}(\cg, C) = \picrit{i}{j}(\cg', C') \,\text{ for all $i \neq j$},
    \]
    i.e., two weighted DAGs have the same critical paths if and only if their maxoids coincide.
\end{lemma}
\begin{proof}

    \enquote{$\impliedby$}: If $(\cg,C)$ and $(\cg', C')$ have the same critical paths then they give rise to the same critical DAG for any $L$, implying equal maxoids.

    \enquote{$\implies$} by contraposition: Suppose that $\pi = \picrit{i}{j}(\cg, C) \neq \picrit{i}{j}(\cg', C') = \pi'$ and denote the nodes on $\pi'$ as follows:
    \begin{align} \label{eqn:pipath}
        \pi':i = \ell'_0 \rightarrow \ell'_1\dots \rightarrow \ell'_{m-1} \rightarrow \ell'_m = j.
    \end{align}
    We may assume that $\pi$ does not contain any of the nodes $\ell_1', \dots, \ell_{m-1}'$ (if this is not the case, then we may replace $j$ with an internal node common to both $\pi$ and $\pi'$ and have shorter but still differing critical paths). Then clearly $\Csepgiven{i}{j}{\ell'_{m-1}}$ holds in $(\cg',C')$ but not in $(\cg, C)$, implying inequality of the respective maxoids.
\end{proof}

\begin{lemma} \label{lem:closurelemma}
Let $(\cg,C)$ be a weighted DAG and $\overline{\cg}$ the transitive closure of~$\cg$. There exists a matrix $\overline{C}$  supported on $\overline{\cg}$ such that
    $\globalCsep{\overline{\cg}}{\overline{C}} = \globalCsep{\cg}{C}$.
\end{lemma}

\begin{proof}
For a fixed $(\cg, C)$ with edge set $E$ , let $\overline{E}$ be the edge set of its transitive closure $\overline{\cg}$. For any two path-connected nodes $i,j \in V$, let $\varepsilon_{ij}$ be the weight of the (not necessarily unique) critical $i-j$ path in $(\cg,C)$, and fix a $ -\infty < \delta < \min_{i,j} \varepsilon_{ij}$. One possible choice of $\overline{C}$ is given by
\begin{align*}
    \overline{C} = (\overline{c}_{ij})_{i,j \in V} = \begin{cases}
        c_{ij} \ \ \text{if}& (i,j) \in E, \\
        \delta \ \ \ \ \text{if}& (i,j) \in \overline{E}\setminus E, \\
        -\infty \ \ &\text{otherwise.}
    \end{cases}
\end{align*}
By construction, no edge in $\overline{E} \setminus E$ is contained in any critical path of $(\overline{\cg}, \overline{C})$. Thus, the statement follows from \Cref{lem:markovcritpaths}.
\end{proof}

\pagebreak

A related notion is the \emph{weighted transitive reduction}.

\begin{definition}
The \emph{weighted transitive reduction} $\cgtr$ of a weighted DAG $(\cg,C)$ is the subgraph of $\cg$ with edges determined as follows:
\begin{center}
    $i \to j \in E(\cgtr)$ $\iff$ $i \to j$ is the unique critical $i-j$ path in $(\cg, C)$.
\end{center}
\end{definition}
\begin{remark} \label{rmk:closuresuffices}
Another consequence of \Cref{lem:markovcritpaths} is that for any weighted DAG $(\cg,C)$ with generic $C$ we have
\begin{align}
    \globalCsep{\cg}{C} = \globalCsep{\cgtr}{C^{\tr}},
\end{align}
 where $C^{\tr}$ is any matrix supported on $\cgtr$ which gives rise to the same critical paths as~$(\cg,C)$. Combined with \Cref{lem:closurelemma}, this means that the maxoid of \emph{any} weighted DAG arises as a maxoid of its transitive closure for an appropriately chosen weight matrix.
\end{remark}

\begin{example} \label{eg:reductionclosurediamond}
The maxoid corresponding to the weighted DAG on the left in \Cref{fig:23-diamond} is
\begin{align*}
    \globalCsep{\cg}{C} = \{ \CI{1,3|2} \ , \ \CI{1,3|2,4} \ , \ \CI{1,4|2} \ , \ \CI{1,4|2,3} \}.
\end{align*}
\begin{minted}{julia-repl}
julia> using Maxoids, Oscar;
julia> E = [[1,2],[1,3],[2,3],[2,4],[3,4]];
julia> G = graph_from_edges(Directed, E);
julia> C = weights_to_tropical_matrix(G, [1,1,1,3,1]);
julia> M = maxoid(G,C)
4-element Vector{CIStmt}:
 [1 _||_ 3 | 2]
 [1 _||_ 3 | {2, 4}]
 [1 _||_ 4 | 2]
 [1 _||_ 4 | {2, 3}]
\end{minted}
This maxoid is also realized by the weighted transitive reduction $\cgtr$ and transitive closure $\overline{\cg}$ when $C^{\tr}$ and $\overline{C}$ are chosen according to \Cref{lem:closurelemma} and \Cref{rmk:closuresuffices}. In this example, $c^{\tr}_{24} > c^{\tr}_{23}+c^{\tr}_{34}$ and $\overline{c}_{14}< \min \{\overline{c}_{12} + \overline{c}_{24} \ , \ \overline{c}_{13} + \overline{c}_{34} \} $ must hold.
\begin{minted}{julia-repl}
julia> G_tr, C_tr = weighted_transitive_reduction(G,C);
julia> Gbar = transitive_closure(G);
julia> Cbar = weights_to_tropical_matrix(Gbar, [1,1,1,1,3,1]);
julia> M == maxoid(G_tr, C_tr) == maxoid(Gbar,Cbar)
true
\end{minted}
\begin{figure}
    \definecolor{bluevar}{rgb} {0.2,0.37,0.64}
    \definecolor{redvar}{rgb}{0.64, 0, 0}
    \centering
    \begin{subfigure}[t]{0.3\linewidth}
    \centering
    \begin{tikzpicture}[inner sep=1pt, scale=0.9]
    \tikzset{
    edge label style/.style={midway, font=\scriptsize\sffamily, text=violet}
}
        \node[draw, circle] (1) at (0,1) {\strut$1$};
        \node[draw, circle] (2) at (-1,0) {\strut$2$};
        \node[draw, circle] (3) at (1,0) {\strut$3$};
        \node[draw, circle] (4) at (0,-1) {\strut$4$};
        \draw[->] (1) -- (2) node[edge label style, above left] {$1$};
        \draw[->] (1) -- (3) node[edge label style, above right] {$1$};
        \draw[->] (2) -- (3) node[edge label style, above, yshift = 1pt] {$1$};
        \draw[->] (2) -- (4) node[edge label style, left,yshift = -2pt] {$3$};
        \draw[->] (3) -- (4) node[edge label style, right, yshift = -2pt] {$1$};
    \end{tikzpicture}
    \captionsetup{labelformat=empty}
    \caption{$(\cg,\textcolor{violet}{C})$}
    \end{subfigure}
    \begin{subfigure}[t]{0.3\linewidth}
    \centering
    \begin{tikzpicture}[inner sep=1pt, scale=0.9]
    \tikzset{
    edge label style/.style={midway, font=\scriptsize\sffamily, text=bluevar}
}
        \node[draw, circle] (1) at (0,1) {\strut$1$};
        \node[draw, circle] (2) at (-1,0) {\strut$2$};
        \node[draw, circle] (3) at (1,0) {\strut$3$};
        \node[draw, circle] (4) at (0,-1) {\strut$4$};
        \draw[->] (1) -- (2) node[edge label style, above left] {$1$};
        \draw[->] (2) -- (3) node[edge label style, above, yshift = 1pt] {$1$};
        \draw[->] (2) -- (4) node[edge label style, left, yshift = -2pt] {$3$};
        \draw[->] (3) -- (4) node[edge label style, right, yshift = -2pt] {$1$};
    \end{tikzpicture}
    \captionsetup{labelformat=empty}
    \caption{$(\cgtr,\textcolor{bluevar}{C^{\tr}})$}
    \end{subfigure}
    \begin{subfigure}[t]{0.3\linewidth}
    \centering
    \begin{tikzpicture}[inner sep=1pt, scale=0.9]
    \tikzset{
    edge label style/.style={midway, font=\scriptsize\sffamily, text=redvar}
}
        \node[draw, circle] (1) at (0,1) {\strut$1$};
        \node[draw, circle] (2) at (-1,0) {\strut$2$};
        \node[draw, circle] (3) at (1,0) {\strut$3$};
        \node[draw, circle] (4) at (0,-1) {\strut$4$};
        \draw[->] (1) -- (2) node[edge label style, above left] {$1$};
        \draw[->] (1) -- (3) node[edge label style, above right] {$1$};
        \draw[->] (1) -- (4) node[edge label style, below right, yshift = -2pt] {$1$};
        \draw[->] (2) -- (3) node[edge label style, above left, yshift = 2pt] {$1$};
        \draw[->] (2) -- (4) node[edge label style, left,yshift = -2pt] {$3$};
        \draw[->] (3) -- (4) node[edge label style, right,yshift = -2pt] {$1$};
    \end{tikzpicture}
    \captionsetup{labelformat=empty}
    \caption{$(\overline{\cg},\textcolor{redvar}{\overline{C}})$}
    \end{subfigure}
    \caption{For appropriate $C^{\tr}$ and $\overline{C}$, $\globalCsep{\cg}{C} = \globalCsep{\cgtr}{C^{\tr}} =  \globalCsep{\overline{\cg}}{\overline{C}}$ holds.}
    \label{fig:23-diamond}
\end{figure}

\end{example}

\begin{theorem}\label{prop:CIstructurecone}
Let $(\cg,C)$ be a weighted DAG with generic $C \in \mathbb{R}^E \setminus \mathcal{H}_\cg$. The set
\begin{align} \label{eqn:conedef}
    \cone_\cg(C) := \big\{ C' \in \mathbb{R}^E \setminus \mathcal{H}_\cg : \globalCsep{\cg}{C} = \globalCsep{\cg}{C'} \big\}
\end{align}
is a full-dimensional open polyhedral cone defined by linear inequalities of the form
    \begin{equation} \label{eqn:coneinequality}
        \omega_{C'}(\picrit{i}{j}(\cg, C)) > \omega_{C'}(\pi), \; \text{ for each $\pi \in P(i,j) \setminus \{\picrit{i}{j}(\cg, C)\}$},
    \end{equation}
    for all distinct $i,j \in V$.
\end{theorem}

\begin{proof}
    By \Cref{lem:markovcritpaths}, the set $\cone_\cg(C)$ consists of all generic weight matrices $C'$ supported on $\cg$ and giving rise to the same critical paths as~$C$. This is precisely what is encoded in the inequalities \eqref{eqn:coneinequality} for all $i,j \in V$.
    These strict linear inequalities in the entries of $C'$ define an open polyhedral cone in $\mathbb{R}^E$ disjoint from $\mathcal{H}_\cg$. The cone is non-empty as $C \in \cone_\cg(C)$ is given, and full-dimensional because $\varepsilon$-perturbations of $C$ in the direction of any $c_{ij}$ preserve its critical~paths.
\end{proof}

\begin{remark} \label{rmk:simplecyclessuffice}
A minimal description of the cone defined in \eqref{eqn:coneinequality} can be obtained by considering only pairs $i,j$ which are connected by multiple \emph{disjoint} paths, in the sense that any two of them form a simple cycle in the skeleton of~$\cg$. Indeed, if two $i-j$ paths $\pi_1$ and $\pi_2$ contain a common intermediate node $k$, then the linear inequality corresponding to the comparison of $\omega_C(\pi_1)$ and $\omega_C(\pi_2)$ is already implied by the linear inequalities which arise from comparing their respective $i-k$ and $k-j$ portions.
\end{remark}

We now study the case where the weight matrix lies on the boundary of a cone. For~generic $C$, let~$\Tilde{C}$ be a matrix lying on a \textit{facet} of the euclidean closure of $\cone_\cg(C)$. This~means that for some pair $i,j \in V$, equality holds in \eqref{eqn:coneinequality} and thus there are two critical $i-j$ paths in $(\cg,\Tilde{C})$: one is the unique critical $i-j$ path $\picrit{i}{j}(\cg,C)$, and the other we denote by~$\pi'$. We assume that the paths are disjoint in the sense of \Cref{rmk:simplecyclessuffice} and that all matrices on the facet of $\cone_\cg(C)$ on which $\Tilde{C}$ lies give rise to the same critical paths as $C$ outside of those which factor through the directed $i-j$ portion of the~DAG.

\begin{theorem} \label{prop:faceunion}
 In the setting described above, the following holds:
    \begin{align} \label{eqn:globalCsepboundary}
        \globalCsep{\cg}{\Tilde{C}} = \globalCsep{\cg}{C} \cup \globalCsep{\cg}{C'},
    \end{align}
    where $C'$ is a matrix supported on $\cg$ giving rise to the same critical paths as $C$ except for in the directed $i-j$ portion, where the unique critical path is $\pi'$.
\end{theorem}

\begin{proof}
    We first consider the simplified case where $\cg$ consists solely of the two directed $i-j$ paths. For readability, we set $\pi:= \picrit{i}{j}(\cg, C)$ and refer to the intermediate nodes of $\pi$~and~$\pi'$ using the notation in \eqref{eqn:pipath}.
    In this setting, $i$~and~$j$ are the only two nodes which are connected by more than one path. Because of this, it suffices to prove both inclusions in \eqref{eqn:globalCsepboundary} only for separation statements of the form $\Csepgiven{i}{j}{L}$.

    \enquote{$\subseteq$}: Let $L \subset V \setminus ij$. Note that if $\CI{i,j|L} \in \globalCsep{\cg}{\Tilde{C}}$ holds, then $L$ intersects $\pi \cup \pi'$ non-trivially. Indeed, if $L \cap (\pi \cup \pi') = \emptyset$, then the critical DAG $\critdag{\Tilde{C},L}$ contains the edge $i \to j$, implying $\ast$-connectedness.     Thus, this choice of $L$ also separates $i$ and $j$ in $(\cg,C)$ or $(\cg, C')$, implying the first inclusion.

    \enquote{$\supseteq$}: In $(\cg, C)$ any $L \subseteq V \setminus ij$ which intersects $\pi$ non-trivially gives rise to the statement $\Csepgiven{i}{j}{L}$. This choice of $L$ also separates $i$ and $j$ in $(\cg, \Tilde{C})$. (Recall that the condition for the edge $i \to j$ to be present in $\critdag{\Tilde{C},L}$ is that \emph{no} critical $i-j$ path in $(\cg, \Tilde{C})$ factors through $L$.) Analogously, any statement of the form $\Csepgiven{i}{j}{L}$ in $\globalCsep{\cg}{C'}$ also holds in $(\cg, \Tilde{C})$. %

    In the more general setting where $\cg$ does not consist solely of directed $i-j$ paths, additional $\ast$-connecting $i-j$ paths may exist. Thus, additional nodes which are not contained in $\pi$ and $\pi'$ may be needed to separate $i$ and $j$. However, these nodes will be required to separate $i$ and $j$ in all three weighted DAGs, since, by our starting assumption, these three matrices give rise to the same critical paths outside of the directed $i-j$ portion of $\cg$. Furthermore, if $i'$ and $j'$ are nodes such that a path between them factors through $\pi$ (and thus also $\pi'$), then a similar argument immediately shows that any $L$ which separates them in $(\cg, \Tilde{C})$ must also separate them in either $(\cg, C)$ or $(\cg, C')$.
    Lastly, any separation which does not involve $\pi$ and $\pi'$ will be present in all three maxoids by assumption and thus the remaining CI statements will be the same as well.
\end{proof}

\begin{remark} \label{rmk:Cs}
In the setting of \Cref{prop:faceunion}, given $\Tilde{C}$ one can obtain a matrix with the properties of $C'$ by replacing $\Tilde{c}_{\ell, \ell'}$ with $\Tilde{c}_{\ell,\ell'} + \varepsilon$, where $(\ell, \ell') \in \pi'$ and $\varepsilon >0$ fulfills
\begin{align} \label{eqn:epsdef}
    \varepsilon < \min_{i',j'\in V} \Biggl\{\min_{\pi_1, \pi_2 \in P(i', j')} |\omega_{\Tilde{C}}(\pi_1) -\omega_{\Tilde{C}}(\pi_2)| \Biggr\}.
\end{align}
This makes $\pi'$ the unique critical $i-j$ path while preserving all other critical paths.
\end{remark}

\Cref{prop:faceunion} implies that facets (and by extension, lower-dimensional faces) of the euclidean closure of $\cone_\cg(C)$ correspond to non-generic maxoids which arise as unions of generic maxoids.

\begin{corollary} \label{cor:fan}
The euclidean closures of the open cones corresponding to the generic maxoids of $\cg$ form a complete polyhedral fan, $\cf_\cg$, in $\mathbb{R}^E$. The maximal cones of $\cf_\cg$ are in bijection with the generic maxoids of $\cg$.
Moreover, the function $\Phi$ which sends a cone of $\cf_\cg$ to its maxoid is an inclusion-reversing surjection:
\begin{align*}
    F_1 \text{ is a face of } F_2 \implies \Phi(F_1) \supseteq \Phi(F_2) \ \ \  \text{for all $F_1, F_2 \in \cf_\cg$. }
\end{align*}
\end{corollary}
We refer to $\mathcal{F}_\cg$ as the \emph{maxoid fan} of $\cg$.
\iffalse
\begin{remark}
\FN{this is wrong. Either modify or remove it.}
It is not hard to show that $\cf_\cg$ is the \emph{Gr\"obner fan} of the ideal \(
I_\cg = \langle\, \sum_{\pi \in P(i, j)} \prod_{e \in \pi} x_e : i, j \in V, |P(i, j)| > 1 \,\rangle\); see \cite{SturmfelsConvex}.
Indeed, any weight matrix $C \in \mathbb{R}^E$ defines a term order which picks out the critical $i-j$ paths of $(\cg, C)$ as the initial term of the generator $f_{ij} = \sum_{\pi \in P(i, j)} \prod_{e \in \pi} x_e$ in~$I_\cg$.
\end{remark}
\fi

\begin{example} \label{eg:diamondfan}
The fan associated to the diamond graph from \Cref{eg:diamond} consists of two maximal cones in $\mathbb{R}^4$ separated by the hyperplane $c_{12}+c_{24} = c_{13}+c_{34}$.
\begin{minted}{julia-repl}
julia> using Maxoids, Oscar;
julia> G = Maxoids.diamond();
julia> F = maxoid_fan(G)
Polyhedral fan in ambient dimension 4
\end{minted}
The corresponding maxoids are
\begin{align*}
    \mathcal{M}_1 &= \{ \CI{2,3|1} \ , \CI{1,4|2,3} \ , \ \CI{1,4|2} \}  &&\text{ for } c_{12}+c_{24} > c_{13}+c_{34}  \\
    \mathcal{M}_2 &= \{ \CI{2,3|1} \ , \CI{1,4|2,3} \ , \ \CI{1,4|3} \}   &&\text{ for } c_{12}+c_{24} < c_{13}+c_{34} \\
    \mathcal{M}_3 &= \{ \CI{2,3|1} \ , \CI{1,4|2,3} \ , \ \CI{1,4|2} \ , \ \CI{1,4|3} \}  &&\text{ for } c_{12}+c_{24} = c_{13}+c_{34}.
\end{align*}
\begin{figure}[h!]
\centering
\begin{subfigure}{0.6\linewidth}
\centering
\begin{tikzpicture}[scale = 0.8]
\definecolor{bluevar}{rgb} {0.2,0.37,0.64}
\definecolor{redvar}{rgb}{0.64, 0, 0}
\definecolor{ltred}{rgb}{0.64, 0, 0}
\tikzset{
    axis/.style = {very thick, ->, >=stealth},
    line/.style = {very thick, color=violet},
    dashed_line/.style = {dashed, very thick, color=redvar},
    red_line/.style = {very thick, color=redvar},
    point/.style = {circle, draw, fill=redvar, inner sep=0pt, minimum size=8pt},
    label/.style = {color=redvar, anchor=south west, font=\bfseries, yshift= - 2pt, xshift = 3pt},
    blue_label/.style = {color=violet, anchor=south west, font=\bfseries},
    blue_dash/.style = {dashed, very thick, color=violet},
    red_dash/.style = {dashed, thick, color=redvar}
}
\tikzstyle{every node}=[font=\large]

\fill[redvar!5] (4,2) -- (-4,2) -- (-4,-2) -- cycle;

\fill[bluevar!5] (4,2) -- (4,-2) -- (-4,-2) -- cycle;

\draw[axis, lightgray] (0, -2) -- (0, 2);
\draw[axis,lightgray] (-4, -0) -- (4,0);

\draw[line] (3.2, 1.6) -- (-3.2, -1.6);
\draw[blue_dash] (3.2,1.6) -- (4, 2);
\draw[blue_dash] (-3.2,-1.6) -- (-4, -2);

\node[text=redvar, font=\bfseries\Large] at (-3, 1.2) {$\mathcal{M}_1$};
\node[text=bluevar, font=\bfseries\Large] at (3, -1.2) {$\mathcal{M}_2$};
\node[text = violet, font =\bfseries\Large] at (3.,1) {$\mathcal{M}_3$};
\end{tikzpicture}
\end{subfigure}
\begin{subfigure}[t]{0.35\linewidth}
\centering
\vspace{-3.2cm}
\definecolor{bluevar}{rgb} {0.2,0.37,0.64}
\definecolor{redvar}{rgb}{0.64, 0, 0}
\definecolor{ltred}{rgb}{0.64, 0, 0}
\begin{tikzpicture}[inner sep=1pt, scale = 1.2]
\node[draw, circle] (1) at (0,1) {\strut$1$};
\node[draw, circle] (2) at (-1,0) {\strut$2$};
\node[draw, circle] (3) at (1,0) {\strut$3$};
\node[draw, circle] (4) at (0,-1) {\strut$4$};
\draw[->, very thick, redvar] (1) -- (2);
\draw[->, very thick, bluevar] (1) -- (3);
\draw[->, very thick, redvar] (2) -- (4);
\draw[->, very thick, bluevar] (3) -- (4);
\end{tikzpicture}
\end{subfigure}
\caption{The maxoid fan $\cf_\cg$ of the diamond, viewed in a 2-dimensional projection of $\mathbb{R}^4$. The full-dimensional cones in the fan correspond to the unique choices of $\picrit{1}{4}$ viewed on the right. The non-generic maxoid $\mathcal{M}_3$ is realized when the weights of the two paths coincide.}
\label{fig:diamondfan}
\end{figure}
\end{example}
The next example shows that non-generic maxoids can coincide with generic maxoids, implying that the function $\Phi$ from \Cref{cor:fan} is generally not a bijection.

\begin{example} \label{eg:3nodeDAG}
Consider the complete topologically ordered DAG on 3 nodes $\cg$ depicted in \Cref{fig:3nodecompleteDAG}. There are two maxoids associated with this DAG:
\begin{figure}[h!]
\centering
\begin{tikzpicture}
    \begin{scope}[->, every node/.style={circle,draw},line width=1pt, node distance=1.8cm]
    \node (1) {$1$};
    \node (2) [right of=1] {$2$};
    \foreach \from/\to in {1/2}
    \draw (\from) -- (\to);
    \path[every node/.style={font=\sffamily\small}]
    (1) -- (2) node [midway, above] {$c_{12}$};
    \node (3) [ right of=2] {$3$};
    \foreach \from/\to in {2/3}
    \draw (\from) -- (\to);
    \path[every node/.style={font=\small}]
    (2) -- (3) node [midway, above] {$c_{23}$};
    \foreach \from/\to in {1/3}
    \draw (\from) to[bend right=40] (\to);
    \path[every node/.style={font=\small}]
    (1) -- (3) node [midway, yshift = -30pt] {$c_{13}$};

\end{scope}
\end{tikzpicture}

\caption[The complete 3-node DAG]{A complete weighted DAG on 3 nodes.}
\label{fig:3nodecompleteDAG}
\end{figure}
\begin{align*}
     \begin{cases}
        \ \bigl\{ \CI{1,3|2}  \bigr\} \ \ \ &\text{if} \ \ c_{13} \leq c_{12}+c_{23} \\
    \ \  \ \ \ \ \emptyset \ \ \ &\text{if} \ \ c_{13} > c_{12}+c_{23}.
    \end{cases}
\end{align*}
\iffalse
The conditional independence statement in the first case arises due to the fact that if the inequality $c_{13} \leq c_{12}c_{23}$ holds, the structural equation for $X_3$ can be rewritten as
\begin{align*}
    X_3 =& \ c_{13}X_1 \vee c_{23}X_2 \vee Z_3 \\
    =& \ c_{13}X_1 \vee c_{23}(c_{12}X_1 \vee Z_2) \vee Z_3 \\
    =& \ (c_{13} \vee c_{12}c_{23})X_1 \vee c_{23}Z_2 \vee Z_3 \\
    =& \  c_{12}c_{23}X_1 \vee c_{23}Z_2 \vee Z_3
    = \ c_{23}X_2 \vee Z_3.
\end{align*}
In particular, $X_1$ influences $X_3$ only indirectly via $X_2$.
\fi
By \Cref{cor:fan}, the maxoid fan $\cf_\cg \subset \mathbb{R}^3$ has 2 maximal cones, the closures of which share the common face $\{c_{13} = c_{12} + c_{23}\}$. However, unlike \Cref{eg:diamondfan}, the non-generic maxoid corresponding to this face is $\{ \CI{1,3|2} \}$ , which coincides with the generic maxoid $\globalCsep{\cg}{C}$ for any $C$ fulfilling $c_{13} < c_{12}+c_{23}$.
\end{example}

\section{The Maxoid Polytope}\label{sec:maxoid-polytope}
In this section, we provide a description of the maxoid fan from \Cref{cor:fan} in terms of a polytope associated to $\cg$, which we call the \emph{maxoid polytope}. To define this polytope, we introduce polynomials with variables indexed by the edges of $\cg$, which lie in the polynomial ring  $\mathbb{R}[E]:= \mathbb{R}[x_{uv} : u\rightarrow v \in E]$.

Let $\cg =  (V,E)$ be a fixed DAG. For $i,j \in V$ such that $|P(i,j)| \geq 1$, consider the \emph{path polynomial}
\begin{equation*}
    f_{ij} = \sum_{\pi \in P(i,j)} \prod_{u \rightarrow v \in \pi}x_{uv} \in \mathbb{R}[E] .
\end{equation*}

The weight of the critical $i-j$ path in $(\cg, C)$ can be computed by evaluating the max-plus \emph{tropicalization} (cf. \Cref{app:trop}) $F_{ij} := Trop(f_{ij})$ of $f_{ij}$ at $C$:
\begin{align*}
    F_{ij}(C) &= \max_{\pi \in P(i,j)} \left\{\sum_{u \to v \in \pi} c_{uv}\right\} =  \max_{\pi \in P(i,j)} \omega_C(\pi) ,
\end{align*}
which is also the $ij$-th entry of the Kleene star of $C$ (see \cref{eqn:kleenestarentry}). \\

The polynomial $F_{ij}$ defines a tropical hypersurface $\mathcal{T}(F_{ij})$, which in turn generates a polyhedral decomposition of $\mathbb{R}^E$. Each term of $F_{ij}$ corresponds uniquely to both a path $\pi \in P(i,j)$, and a full-dimensional cell of $\mathcal{T}(F_{ij})$. For generic $C$ it holds that
\begin{center}
    $\picrit{i}{j}(\cg,C) = \pi \iff $ $C$ lies in the interior of the region of $\mathcal{T}(F_{ij})$ corresponding to $\pi$.
\end{center}
As $F_{ij}$ has constant coefficients, \Cref{thm:jos} implies that the subdivision induced by $\mathcal{T}(F_{ij})$ has a particularly nice form: it is the normal fan of the Newton polytope of $f_{ij}$ (see \Cref{app:polytopes}). The \emph{common refinement} of all such normal fans $\mathcal{N}(\Newt(f_{ij}))$ as one varies over the pairs $i,j$ fully encodes the combinatorics of the critical paths of $\cg$. As the common refinement of the normal fans is the normal fan of the Minkowski sum of the individual polytopes (cf. \Cref{thm:normalminkowskisum}), these observations can be summed up as follows:

\begin{theorem} \label{thm:maxoidpolytope}
    The maxoid fan \(\mathcal{F}_\cg\) of $\cg$ coincides with the normal fan of the polytope~$P_\cg$, where
    \[ P_\cg := \sum_{\substack{i,j \in V, \\  |P(i,j)| \geq 1}}\Newt(f_{ij}).\]
    In particular, this means that \(P_\cg\) is the normal fan of the Newton polytope \(\Newt\left(\prod f_{ij}\right)\) of the product
    of path polynomials \(f_{ij}\) with $|P(i,j)| \geq 1$.
\end{theorem}
Standard results from polyhedral geometry yield the following correspondence:
\begin{corollary} \label{cor:maxoidpolytopevertices}
    The vertices of $P_\cg$ are in one-to-one correspondence with the generic maxoids of $\cg$. Lower-dimensional faces correspond to non-generic maxoids.
\end{corollary}
\begin{example} \label{eg:Complete4}
Let $\cg$ be the complete topologically ordered DAG on 4 nodes, which is supported by generic matrices $C \in \mathbb{R}^E \cong \mathbb{R}^6$ of the form depicted below. \\[.5em]
\noindent\begin{minipage}{\textwidth}
\centering
\begin{minipage}[c]{\dimexpr0.3\textwidth}
\begin{tikzpicture}[inner sep=1pt, scale=0.9]
\node[draw, circle] (1) at (0,1) {\strut$1$};
\node[draw, circle] (2) at (-1,0) {\strut$2$};
\node[draw, circle] (3) at (1,0) {\strut$3$};
\node[draw, circle] (4) at (0,-1) {\strut$4$};
\draw[->] (1) -- (2);
\draw[->] (1) -- (3);
\draw[->] (2) -- (4);
\draw[->] (3) -- (4);
\draw[->] (1) -- (4);
\draw[->] (2) -- (3);
\end{tikzpicture}
\end{minipage}
\begin{minipage}[c]{\dimexpr0.3\textwidth}
\vfil
\[
  C = \begin{pmatrix}
  -\infty & c_{12} & c_{13} & c_{14} \\
  -\infty & -\infty & c_{24} & c_{24} \\
  -\infty & -\infty & -\infty & c_{34}\\
  -\infty & -\infty & -\infty & -\infty
  \end{pmatrix}
\]
\vfil
\end{minipage}
\end{minipage} \\[.5em]

The critical path structure (and by \Cref{lem:markovcritpaths}, $\globalCsep{\cg}{C}$) of a fixed $(\cg,C)$ is fully determined by which terms of the following expressions attain their respective maxima:
\begin{align*}
    F_{13}(C) &= \max \{c_{12}+c_{23}, c_{23}\}, \\
    F_{24}(C) &= \max \{c_{23}+c_{34}, c_{24}\}, \\
    F_{14}(C) &= \max \{c_{14}, c_{12}+c_{24}, c_{13}+c_{34}, c_{12}+c_{23} + c_{34} \}.
\end{align*}
For generic $C$, this specifies the full-dimensional cone of $\cf_\cg$ in which it lies. In this example, the fan $\cf_\cg \subset \mathbb{R^6}$ has 9 full-dimensional cones and a 3-dimensional lineality space. The maxoid polytope $P_\cg$ is a 3-dimensional polytope in $\mathbb{R}^6$, the vertices of which correspond bijectively to the 9 generic maxoids of $\cg$.

\begin{figure}[h!]
    \centering
    \label{fig:ex33polytope}
\begin{tikzpicture}[
  x = {(-1.2cm, 0.1cm)},
  y = {(0.9cm, -1.1cm)},
  z = {(-1cm, -1.5cm)},
  scale = 1.2,
  color = {lightgray}]

  \coordinate (v0_unnamed__1) at (-1, 0, -1);
  \coordinate (v1_unnamed__1) at (2, 2, -1);
  \coordinate (v2_unnamed__1) at (2, -1, 2);
  \coordinate (v3_unnamed__1) at (0, -1, 0);
  \coordinate (v4_unnamed__1) at (1, 2, -1);
  \coordinate (v5_unnamed__1) at (0, 0, -1);
  \coordinate (v6_unnamed__1) at (-1, 0, 1);
  \coordinate (v7_unnamed__1) at (-1, -1, 0);
  \coordinate (v8_unnamed__1) at (-1, -1, 2);

  \definecolor{vertexcolor_unnamed__1}{rgb}{ 0 0 0 }

  \definecolor{maxoid_1}{RGB}{51,34,136}
  \definecolor{maxoid_2}{RGB}{17,119,51}
  \definecolor{maxoid_3}{RGB}{68,170,153}
  \definecolor{maxoid_4}{RGB}{136,204,238}
  \definecolor{maxoid_5}{RGB}{221,204,119}
  \colorlet{maxoid_6}{orange!65}
  \definecolor{maxoid_7}{RGB}{204,102,119}
  \definecolor{maxoid_8}{RGB}{170,68,153}
  \definecolor{maxoid_9}{RGB}{136,34,85}

  \tikzstyle{maxoid_label} = [inner sep=2pt, rectangle, rounded corners=3pt]

  \tikzstyle{vertexstyle_unnamed__1_0} = [circle, scale=0.25, fill=maxoid_9]
  \tikzstyle{vertexstyle_unnamed__1_1} = [circle, scale=0.25, fill=maxoid_6]
  \tikzstyle{vertexstyle_unnamed__1_2} = [circle, scale=0.25, fill=maxoid_1]
  \tikzstyle{vertexstyle_unnamed__1_3} = [circle, scale=0.25, fill=maxoid_3]
  \tikzstyle{vertexstyle_unnamed__1_4} = [circle, scale=0.25, fill=maxoid_7]
  \tikzstyle{vertexstyle_unnamed__1_5} = [circle, scale=0.25, fill=maxoid_8]
  \tikzstyle{vertexstyle_unnamed__1_6} = [circle, scale=0.25, fill=maxoid_5]
  \tikzstyle{vertexstyle_unnamed__1_7} = [circle, scale=0.25, fill=maxoid_4]
  \tikzstyle{vertexstyle_unnamed__1_8} = [circle, scale=0.25, fill=maxoid_2]

  \definecolor{facetcolor_unnamed__1}{rgb}{ 0.8 0.8 0.8 }

  \definecolor{edgecolor_unnamed__1}{rgb}{ 0 0 0 }
  \tikzstyle{facetstyle_unnamed__1} = [fill=facetcolor_unnamed__1, fill opacity=0.2]

  \begin{scope}[blend group=darken]

    \fill[facetstyle_unnamed__1] (v4_unnamed__1) -- (v1_unnamed__1) -- (v2_unnamed__1) -- (v8_unnamed__1) -- (v6_unnamed__1) -- cycle; %
    \fill[facetstyle_unnamed__1] (v0_unnamed__1) -- (v7_unnamed__1) -- (v3_unnamed__1) -- (v5_unnamed__1) -- cycle; %
    \fill[facetstyle_unnamed__1] (v2_unnamed__1) -- (v3_unnamed__1) -- (v7_unnamed__1) -- (v8_unnamed__1) -- cycle; %
    \fill[facetstyle_unnamed__1] (v7_unnamed__1) -- (v0_unnamed__1) -- (v6_unnamed__1) -- (v8_unnamed__1) -- cycle; %
    \fill[facetstyle_unnamed__1] (v2_unnamed__1) -- (v3_unnamed__1) -- (v5_unnamed__1) -- (v1_unnamed__1) -- cycle; %
    \fill[facetstyle_unnamed__1] (v0_unnamed__1) -- (v4_unnamed__1) -- (v6_unnamed__1) -- cycle; %
    \fill[facetstyle_unnamed__1] (v0_unnamed__1) -- (v4_unnamed__1) -- (v1_unnamed__1) -- (v5_unnamed__1) -- cycle; %

  \draw[facetstyle_unnamed__1] (v2_unnamed__1) -- (v1_unnamed__1) -- (v4_unnamed__1) -- (v6_unnamed__1) -- (v8_unnamed__1) -- (v2_unnamed__1) -- cycle;

    \draw[line width=2pt,maxoid_1] (v2_unnamed__1) -- (v3_unnamed__1);
    \draw[line width=2pt,maxoid_1] (v2_unnamed__1) -- (v1_unnamed__1);
    \draw[line width=2pt,maxoid_1] (v2_unnamed__1) -- (v8_unnamed__1);

    \draw[line width=2pt,maxoid_2] (v7_unnamed__1) -- (v8_unnamed__1);
    \draw[line width=2pt,maxoid_2] (v6_unnamed__1) -- (v8_unnamed__1);

    \draw[line width=2pt,maxoid_3] (v3_unnamed__1) -- (v7_unnamed__1);
    \draw[line width=2pt,maxoid_3] (v3_unnamed__1) -- (v5_unnamed__1);

    \draw[line width=2pt,maxoid_4] (v7_unnamed__1) -- (v0_unnamed__1);

    \draw[line width=2pt,maxoid_5] (v6_unnamed__1) -- (v4_unnamed__1);
    \draw[line width=2pt,maxoid_5] (v6_unnamed__1) -- (v0_unnamed__1);

    \draw[line width=2pt,maxoid_6] (v1_unnamed__1) -- (v4_unnamed__1);
    \draw[line width=2pt,maxoid_6] (v1_unnamed__1) -- (v5_unnamed__1);

    \draw[line width=2pt,maxoid_7] (v4_unnamed__1) -- (v0_unnamed__1);

    \draw[line width=2pt,maxoid_8] (v5_unnamed__1) -- (v0_unnamed__1);
  \end{scope}

  \foreach \i in {7,3,0,5} {
    \node at (v\i_unnamed__1) [vertexstyle_unnamed__1_\i] {};
  }

  \foreach \i in {8,6,2,4,1} {
    \node at (v\i_unnamed__1) [vertexstyle_unnamed__1_\i] {};
  }

  \node[maxoid_label, left       =1mm of v2_unnamed__1, text=white, fill=maxoid_1, draw=maxoid_1] {1};
  \node[maxoid_label, below      =1mm of v8_unnamed__1, text=white, fill=maxoid_2, draw=maxoid_2] {2};
  \node[maxoid_label, above      =1mm of v3_unnamed__1, text=white, fill=maxoid_3, draw=maxoid_3] {3};
  \node[maxoid_label, below right=1mm of v7_unnamed__1, text=black, fill=maxoid_4, draw=maxoid_4] {4};
  \node[maxoid_label, below      =1mm of v6_unnamed__1, text=black, fill=maxoid_5, draw=maxoid_5] {5};
  \node[maxoid_label, below      =1mm of v1_unnamed__1, text=black, fill=maxoid_6, draw=maxoid_6] {6};
  \node[maxoid_label, below right=1mm of v4_unnamed__1, text=white, fill=maxoid_7, draw=maxoid_7] {7};
  \node[maxoid_label, above      =1mm of v5_unnamed__1, text=white, fill=maxoid_8, draw=maxoid_8] {8};
  \node[maxoid_label, above right=1mm of v0_unnamed__1, text=white, fill=maxoid_9, draw=maxoid_9] {9};

\end{tikzpicture}
\caption{The maxoid polytope for the complete graph in \Cref{eg:Complete4}. The colored vertex labels refer to the types of maxoids in \Cref{fig:hasse-diagram-maxoids}.}
\end{figure}

\Cref{fig:hasse-diagram-maxoids} depicts the entire face lattice of $P_\cg$, which is dual to the lattice of the fan~$\cf_\cg$. The elements of the lattice (corresponding to the faces of $P_\cg$) are color-coded by the maxoids they give rise to, which are depicted above the Hasse diagram. By \Cref{prop:faceunion}, the non-generic maxoid arising at a face of dimension $\geq 1$ of $P_\cg$ is the union of the generic maxoids corresponding to the vertices contained in said face. In this particular example, any collection of generic maxoids lying on a common face is ordered linearly with respect to inclusion. This implies that the maxoid corresponding to any face is the largest maxoid arising from the vertices it contains.

\begin{figure}
\centering
\begin{tikzpicture}[yscale=1.5]
  \definecolor{maxoid_1}{RGB}{51,34,136}
  \definecolor{maxoid_2}{RGB}{17,119,51}
  \definecolor{maxoid_3}{RGB}{68,170,153}
  \definecolor{maxoid_4}{RGB}{136,204,238}
  \definecolor{maxoid_5}{RGB}{221,204,119}
  \colorlet{maxoid_6}{orange!65}
  \definecolor{maxoid_7}{RGB}{204,102,119}
  \definecolor{maxoid_8}{RGB}{170,68,153}
  \definecolor{maxoid_9}{RGB}{136,34,85}

  \tikzstyle{maxoid_1} = [text=white, fill=maxoid_1, draw=maxoid_1, inner sep=2pt, rectangle, rounded corners=3pt]
  \tikzstyle{maxoid_2} = [text=white, fill=maxoid_2, draw=maxoid_2, inner sep=2pt, rectangle, rounded corners=3pt]
  \tikzstyle{maxoid_3} = [text=white, fill=maxoid_3, draw=maxoid_3, inner sep=2pt, rectangle, rounded corners=3pt]
  \tikzstyle{maxoid_4} = [text=black, fill=maxoid_4, draw=maxoid_4, inner sep=2pt, rectangle, rounded corners=3pt]
  \tikzstyle{maxoid_5} = [text=black, fill=maxoid_5, draw=maxoid_5, inner sep=2pt, rectangle, rounded corners=3pt]
  \tikzstyle{maxoid_6} = [text=black, fill=maxoid_6, draw=maxoid_6, inner sep=2pt, rectangle, rounded corners=3pt]
  \tikzstyle{maxoid_7} = [text=white, fill=maxoid_7, draw=maxoid_7, inner sep=2pt, rectangle, rounded corners=3pt]
  \tikzstyle{maxoid_8} = [text=white, fill=maxoid_8!85, draw=maxoid_8!85, inner sep=2pt, rectangle, rounded corners=3pt]
  \tikzstyle{maxoid_9} = [text=white, fill=maxoid_9, draw=maxoid_9, inner sep=2pt, rectangle, rounded corners=3pt]

  \node[maxoid_1] (1) at (0,0) {1};
  \node[maxoid_6] (6) at (-3,-1) {6};
  \node[maxoid_5] (5) at (-1,-1) {5};
  \node[maxoid_3] (3) at (1,-1) {3};
  \node[maxoid_2] (2) at (3,-1) {2};
  \node[maxoid_8] (8) at (-2,-2) {8};
  \node[maxoid_7] (7) at (0,-2) {7};
  \node[maxoid_4] (4) at (2,-2) {4};
  \node[maxoid_9] (9) at (0,-3) {9};

  \draw (1) -- (6);
  \draw (1) -- (5);
  \draw (1) -- (3);
  \draw (1) -- (2);

  \draw (6) -- (8);
  \draw (6) -- (7);
  \draw (5) -- (7);
  \draw (3) -- (8);
  \draw (3) -- (4);
  \draw (2) -- (7);
  \draw (2) -- (4);

  \draw (8) -- (9);
  \draw (7) -- (9);
  \draw (4) -- (9);
\end{tikzpicture}
\caption{The maxoid poset of the complete graph. Each node corresponds to one of the maxoids listed in \Cref{fig:hasse-diagram-maxoids}.}
\label{fig:maxoid-poset}
\end{figure}
The maxoid poset of $\cg$ is a coarsening of the face poset of $P_\cg$ obtained by identifying nodes of the same color in \Cref{fig:hasse-diagram-maxoids}. It is shown separately in \Cref{fig:maxoid-poset} for clarity. This~poset is, in this case, a graded lattice but it is not polytopal because the interval between \tikz[baseline]{\definecolor{maxoid_1}{RGB}{51,34,136}\node[yshift=3pt, text=white, fill=maxoid_1, draw=maxoid_1, inner sep=2pt, rectangle, rounded corners=3pt] {$1$}} and \tikz[baseline]{\definecolor{maxoid_7}{RGB}{204,102,119}\node[yshift=3pt, text=white, fill=maxoid_7, draw=maxoid_7, inner sep=2pt, rectangle, rounded corners=3pt] {$7$}} violates the \emph{diamond property} \cite[Theorem~2.7]{Ziegler}.
\end{example}

\begin{figure}
    \centering
    % polymake for ferry
% Thu Nov 20 12:05:10 2025
% unnamed

\begin{tikzpicture}[x  = {(1em, 0em)},
                    y  = {(0em, 10em)},
                    scale = .95,
                    color = {lightgray},
                    rotate=90,transform shape]

  % DEF COORDINATES
  \coordinate (v0_unnamed__1) at (0, 3);
  
  \coordinate (v1_unnamed__1) at (26, 2);
  \coordinate (v2_unnamed__1) at (19.5, 2);
  \coordinate (v3_unnamed__1) at (13, 2);
  \coordinate (v4_unnamed__1) at (7, 2);
  \coordinate (v5_unnamed__1) at (1, 2);
  \coordinate (v6_unnamed__1) at (-5.8, 2);
  \coordinate (v7_unnamed__1) at (-12.2, 2);
  \coordinate (v8_unnamed__1) at (-18.5, 2);
  \coordinate (v9_unnamed__1) at (-25, 2);
  
  \coordinate (v10_unnamed__1) at (28, 1);
  \coordinate (v11_unnamed__1) at (24, 1);
  \coordinate (v12_unnamed__1) at (20, 1);
  \coordinate (v13_unnamed__1) at (15, 1);
  \coordinate (v14_unnamed__1) at (11, 1);
  \coordinate (v15_unnamed__1) at (7, 1);
  \coordinate (v16_unnamed__1) at (3, 1);
  \coordinate (v17_unnamed__1) at (-1.5, 1);
  \coordinate (v18_unnamed__1) at (-5.25789, 1);
  \coordinate (v19_unnamed__1) at (-8.45789, 1);
  \coordinate (v20_unnamed__1) at (-11.6579, 1);
  \coordinate (v21_unnamed__1) at (-14.8579, 1);
  \coordinate (v22_unnamed__1) at (-18.0579, 1);
  \coordinate (v23_unnamed__1) at (-21.2579, 1);
  
  \coordinate (v24_unnamed__1) at (25, 0);
  \coordinate (v25_unnamed__1) at (19, 0);
  \coordinate (v26_unnamed__1) at (13, 0);
  \coordinate (v27_unnamed__1) at (7, 0);
  \coordinate (v28_unnamed__1) at (1, 0);
  \coordinate (v29_unnamed__1) at (-7, 0);
  \coordinate (v30_unnamed__1) at (-15, 0);
  
  \coordinate (v31_unnamed__1) at (0, -1);

  % VERTEXCOLOR
  \definecolor{vertexcolor_unnamed__1_0}{rgb}{ 0 0 0 }
  \definecolor{vertexcolor_unnamed__1_1}{rgb}{ 1 1 1 }

  \definecolor{maxoid_1}{RGB}{51,34,136}
  \definecolor{maxoid_2}{RGB}{17,119,51}
  \definecolor{maxoid_3}{RGB}{68,170,153}
  \definecolor{maxoid_4}{RGB}{136,204,238}
  \definecolor{maxoid_5}{RGB}{221,204,119}
  \colorlet{maxoid_6}{orange!65}
  \definecolor{maxoid_7}{RGB}{204,102,119}
  \definecolor{maxoid_8}{RGB}{170,68,153}
  \definecolor{maxoid_9}{RGB}{136,34,85}

  % VERTEXBORDERCOLOR
  \definecolor{vertexbordercolor_unnamed__1}{rgb}{ 0 0 0 }

  \tikzstyle{maxoid_1} = [text=white, fill=maxoid_1, draw=maxoid_1]
  \tikzstyle{maxoid_2} = [text=white, fill=maxoid_2, draw=maxoid_2]
  \tikzstyle{maxoid_3} = [text=white, fill=maxoid_3, draw=maxoid_3]
  \tikzstyle{maxoid_4} = [text=black, fill=maxoid_4, draw=maxoid_4]
  \tikzstyle{maxoid_5} = [text=black, fill=maxoid_5, draw=maxoid_5]
  \tikzstyle{maxoid_6} = [text=black, fill=maxoid_6, draw=maxoid_6]
  \tikzstyle{maxoid_7} = [text=white, fill=maxoid_7, draw=maxoid_7]
  \tikzstyle{maxoid_8} = [text=white, fill=maxoid_8!85, draw=maxoid_8!85]
  \tikzstyle{maxoid_9} = [text=white, fill=maxoid_9, draw=maxoid_9]

  % DEF VERTEXSTYLES
  % \tikzstyle{vertexstyle_unnamed__1_0} = [text=black, inner sep=2pt, rectangle, rounded corners=3pt,fill=vertexcolor_unnamed__1_0, draw=vertexbordercolor_unnamed__1,]
  
  \tikzstyle{vertexstyle_unnamed__1_1} = [inner sep=2pt, rectangle, rounded corners=3pt, maxoid_1,]
  \tikzstyle{vertexstyle_unnamed__1_2} = [inner sep=2pt, rectangle, rounded corners=3pt, maxoid_2,]
  \tikzstyle{vertexstyle_unnamed__1_3} = [inner sep=2pt, rectangle, rounded corners=3pt, maxoid_3,]
  \tikzstyle{vertexstyle_unnamed__1_4} = [inner sep=2pt, rectangle, rounded corners=3pt, maxoid_4,]
  \tikzstyle{vertexstyle_unnamed__1_5} = [inner sep=2pt, rectangle, rounded corners=3pt, maxoid_5,]
  \tikzstyle{vertexstyle_unnamed__1_6} = [inner sep=2pt, rectangle, rounded corners=3pt, maxoid_6,]
  \tikzstyle{vertexstyle_unnamed__1_7} = [inner sep=2pt, rectangle, rounded corners=3pt, maxoid_7,]
  \tikzstyle{vertexstyle_unnamed__1_8} = [inner sep=2pt, rectangle, rounded corners=3pt, maxoid_8,]
  \tikzstyle{vertexstyle_unnamed__1_9} = [inner sep=2pt, rectangle, rounded corners=3pt, maxoid_9,]
   
  \tikzstyle{vertexstyle_unnamed__1_10} = [inner sep=2pt, rectangle, rounded corners=3pt, maxoid_1,]
  \tikzstyle{vertexstyle_unnamed__1_11} = [inner sep=2pt, rectangle, rounded corners=3pt, maxoid_1,]
  \tikzstyle{vertexstyle_unnamed__1_12} = [inner sep=2pt, rectangle, rounded corners=3pt, maxoid_1,]
  \tikzstyle{vertexstyle_unnamed__1_13} = [inner sep=2pt, rectangle, rounded corners=3pt, maxoid_2,]
  \tikzstyle{vertexstyle_unnamed__1_14} = [inner sep=2pt, rectangle, rounded corners=3pt, maxoid_2,]
  \tikzstyle{vertexstyle_unnamed__1_15} = [inner sep=2pt, rectangle, rounded corners=3pt, maxoid_3,]
  \tikzstyle{vertexstyle_unnamed__1_16} = [inner sep=2pt, rectangle, rounded corners=3pt, maxoid_3,]
  \tikzstyle{vertexstyle_unnamed__1_17} = [inner sep=2pt, rectangle, rounded corners=3pt, maxoid_4,]
  \tikzstyle{vertexstyle_unnamed__1_18} = [inner sep=2pt, rectangle, rounded corners=3pt, maxoid_5,]
  \tikzstyle{vertexstyle_unnamed__1_19} = [inner sep=2pt, rectangle, rounded corners=3pt, maxoid_5,]
  \tikzstyle{vertexstyle_unnamed__1_20} = [inner sep=2pt, rectangle, rounded corners=3pt, maxoid_6,]
  \tikzstyle{vertexstyle_unnamed__1_21} = [inner sep=2pt, rectangle, rounded corners=3pt, maxoid_6,]
  \tikzstyle{vertexstyle_unnamed__1_22} = [inner sep=2pt, rectangle, rounded corners=3pt, maxoid_7,]
  \tikzstyle{vertexstyle_unnamed__1_23} = [inner sep=2pt, rectangle, rounded corners=3pt, maxoid_8,]
  
  \tikzstyle{vertexstyle_unnamed__1_24} = [inner sep=2pt, rectangle, rounded corners=3pt, maxoid_1,]
  \tikzstyle{vertexstyle_unnamed__1_25} = [inner sep=2pt, rectangle, rounded corners=3pt, maxoid_1,]
  \tikzstyle{vertexstyle_unnamed__1_26} = [inner sep=2pt, rectangle, rounded corners=3pt, maxoid_1,]
  \tikzstyle{vertexstyle_unnamed__1_27} = [inner sep=2pt, rectangle, rounded corners=3pt, maxoid_2,]
  \tikzstyle{vertexstyle_unnamed__1_28} = [inner sep=2pt, rectangle, rounded corners=3pt, maxoid_3,]
  \tikzstyle{vertexstyle_unnamed__1_29} = [inner sep=2pt, rectangle, rounded corners=3pt, maxoid_5,]
  \tikzstyle{vertexstyle_unnamed__1_30} = [inner sep=2pt, rectangle, rounded corners=3pt, maxoid_6,]
  
  % \tikzstyle{vertexstyle_unnamed__1_31} = [text=black, inner sep=2pt, rectangle, rounded corners=3pt,fill=vertexcolor_unnamed__1_1, draw=vertexbordercolor_unnamed__1,]

  % EDGECOLOR
  \definecolor{edgecolor_unnamed__1}{rgb}{ 0 0 0 }
  \tikzstyle{edgestyle_unnamed__1} = [thick,color=edgecolor_unnamed__1]

  % EDGES

  \foreach \i/\k in {10/1,10/2,11/1,11/3,12/1,12/6,13/2,13/4,14/2,14/5,15/3,15/4,16/3,16/8,17/4,17/9,18/5,18/7,19/5,19/9,20/6,20/7,21/6,21/8,22/7,22/9,23/8,23/9,24/10,24/12,24/14,24/18,24/20,25/10,25/11,25/13,25/15,26/11,26/12,26/16,26/21,27/13,27/14,27/17,27/19,28/15,28/16,28/17,28/23,29/18,29/19,29/22,30/20,30/21,30/22,30/23} {
   \draw[edgestyle_unnamed__1] (v\i_unnamed__1) -- (v\k_unnamed__1);
  }

  % POINTS
  \foreach \i/\label in {
    1/1,2/2,3/3,4/4,5/5,6/6,7/7,8/8,9/9,
    10/1 2,11/1 3,12/1 6,13/2 4,14/2 5,15/3 4,16/3 8,17/4 9,18/5 7,19/5 9,20/6 7,21/6 8,22/7 9,23/8 9,
    24/1 2 5 6 7,25/1 2 3 4,26/1 3 6 8,27/2 4 5 9,28/3 4 8 9,29/5 7 9,30/6 7 8 9
  } {
    \node at (v\i_unnamed__1) [vertexstyle_unnamed__1_\i] {\label};
  }

  \node[above=15mm of v1_unnamed__1, inner sep=2pt, rectangle, rounded corners=3pt,draw=maxoid_1,text=maxoid_1,align=center] {
    $\CI{1,3|2}$,\\$\CI{1,3|2,4}$,\\$\CI{1,4|2}$,\\$\CI{1,4|3}$,\\$\CI{1,4|2,3}$,\\$\CI{2,4|3}$,\\$\CI{2,4|1,3}$  
  };

  \node[above=15mm of v2_unnamed__1, inner sep=2pt, rectangle, rounded corners=3pt,draw=maxoid_2,text=maxoid_2,align=center] {
    $\CI{1,3|2}$,\\$\CI{1,3|2,4}$,\\$\CI{1,4|2}$,\\$\CI{1,4|2,3}$
  };

  \node[above=15mm of v3_unnamed__1, inner sep=2pt, rectangle, rounded corners=3pt,draw=maxoid_3,text=maxoid_3,align=center] {
    $\CI{1,3|2}$,\\$\CI{2,4|1,3}$
  };

  \node[above=15mm of v4_unnamed__1, inner sep=2pt, rectangle, rounded corners=3pt,draw=maxoid_4!75!black,text=maxoid_4!75!black] {
    $\CI{1,3|2}$
  };

  \node[above=15mm of v5_unnamed__1, inner sep=2pt, rectangle, rounded corners=3pt,draw=maxoid_5!50!black,text=maxoid_5!50!black,align=center] {
    $\CI{1,4|2}$,\\$\CI{1,4|2,3}$
  };

  \node[above=15mm of v6_unnamed__1, inner sep=2pt, rectangle, rounded corners=3pt,draw=maxoid_6!50!black,text=maxoid_6!50!black,align=center] {
    $\CI{1,4|3}$,\\$\CI{1,4|2,3}$,\\$\CI{2,4|3}$,\\$\CI{2,4|1,3}$
  };

  \node[above=15mm of v7_unnamed__1, inner sep=2pt, rectangle, rounded corners=3pt,draw=maxoid_7,text=maxoid_7] {
    $\CI{1,4|2,3}$
  };

  \node[above=15mm of v8_unnamed__1, inner sep=2pt, rectangle, rounded corners=3pt,draw=maxoid_8,text=maxoid_8] {
    $\CI{2,4|1,3}$
  };

  \node[above=15mm of v9_unnamed__1, inner sep=2pt, rectangle, rounded corners=3pt,draw=maxoid_9,text=maxoid_9] {
    $\varnothing$
  };

\end{tikzpicture}
    \caption{Face lattice of \(P_\cg\) with vertices at the top, colored by type of maxoid. By \Cref{prop:faceunion}, each face of dimension $\geq 1$ inherits the color of the largest generic maxoid it contains.}
    \label{fig:hasse-diagram-maxoids}
\end{figure}

An application of \Cref{cor:maxoidpolytopevertices} is illustrated in the two tables below. \Cref{tab:totalmaxoids} counts the number of distinct generic and non-generic maxoids which are realized by some topologically ordered, transitively closed DAG on $n$ non-isolated nodes. (By \Cref{lem:closurelemma}, these are complete enumerations of maxoids on $n$ nodes up to relabeling of vertices.)

\begin{minted}{julia-repl}
julia> using Maxoids, Oscar;
julia> GG = all_top_ordered_TDAGs(4);
julia> length(all_maxoids(collect(GG))
41
julia> length(all_maxoids(collect(GG), generic_only = true))
40
\end{minted}
\begin{table}[h!]
    \centering
    \begin{tabular}{cccc}
         \#$V $& \#TDAGs & \#distinct maxoids & \#distinct generic maxoids \\
         \hline
         3 & 3 & 4 & 4 \\
         4 & 18 & 41 & 40 \\
         5 & 181 & 987 & 892 \\
         6 & 32768 & $-$ & $-$
    \end{tabular}
    \caption{Enumeration of all distinct generic and non-generic maxoids arising from topologically ordered (TDAGs) on 3, 4 and 5 nodes.}
    \label{tab:totalmaxoids}
\end{table}

    The entries for 6 nodes are missing because enumerating the distinct maxoids in this way requires their explicit computation, and doing this for all 32768 TDAGs on 6 nodes did not terminate in under 24 hours. On the other hand, enumerating the distinct generic maxoids on a specific DAG $\cg$ is possible without computing them explicitly; by \Cref{cor:maxoidpolytopevertices} it suffices to construct $P_\cg$ and count its vertices. In \Cref{tab:completeDAGmaxoids} we do this for the complete topologically ordered DAG (i.e. the DAG with edge set $\{ (i,j) : 1 \leq i < j \leq n \}$) on $n$ nodes for $ 3\leq n \leq 6$, as well as providing information about the dimension of the polytope.

\begin{minted}{julia-repl}
julia> using Maxoids, Oscar;
julia> complete_DAG(4) |> maxoid_polytope |> f_vector |> transpose
1×3 transpose(::Vector{ZZRingElem}) with eltype ZZRingElem:
 9  14  7
\end{minted}

\begin{table}[h!]
    \centering
    \begin{tabular}{cccc}
         \#$V $& \#$E$ & $\dim(P_\cg)$ & \#vertices $P_\cg$ \\
         \hline
         3 & 3 & 1 & 2 \\
         4 & 6 & 3 & 9 \\
         5 & 10 & 6 & 103 \\
         6 & 15 & 10 & 3324
    \end{tabular}
    \caption{Table of values for the complete topologically ordered DAGs on 3,4,5, and 6 nodes. The cardinality of the edge set \#$E$ is the ambient dimension of the polytope, whereas the number of vertices of $P_\cg$ is the number of distinct generic maxoids on $\cg$. }
    \label{tab:completeDAGmaxoids}
\end{table}

The realization of the maxoid fan as the normal fan of a Newton polytope is very close in construction to the Gr\"obner fan
of an ideal. This raises the question whether we can realize the maxoid fan as a Gr\"obner fan. This is not always possible
and we present one obstruction to the maxoid fan being a Gr\"obner fan.

\begin{definition}
  Let \(\cg = (V, E)\) be a DAG and consider the polynomial ring $k[E] \defas k[x_e : e \in E]$ with one indeterminate per edge over a field~$k$. If $\pi$ is a path (or any sequence of edges) in $\cg$, then $x_\pi \defas \prod_{e \in \pi} x_e$ is the associated \emph{path monomial}. The \emph{binomial critical path ideal} of $\cg$ is the ideal of $k[E]$ given by
  \[
    J_\cg = \left\langle\, x_\pi - x_{\pi'} : \pi,\pi'\in P(i,j)\ \text{for all } i\neq j \,\right\rangle.
  \]
\end{definition}

The lattice ideal \(J_\cg\) encodes the linear inequalities from \Cref{prop:CIstructurecone} as binomials. This appears
as a reasonable construction since we also have the following relationship between the maxoid fan and the Gr\"obner fan of \(J_\cg\).

\begin{lemma}
    The Gr\"obner fan of \(J_\cg\) is a refinement of the maxoid fan.

    \begin{proof}
        Let \(C,C'\in\mathbb{R}^E\) be two weight matrices for \(\cg\) in the same full-dimensional
        cone of the Gr\"obner cones of \(J_\cg\). Then, we know that \(\initial(J_\cg) = \initial[C'](J_\cg)\) and
        in particular \(\initial(f) = \initial[C'](f)\) for every generator \(f\) of \(J_\cg\).

        The choice of leading term \(\initial(f)\) of \(f = x_\pi - x_{\pi'}\) depends on the linear inequality \[
            \omega_C(\pi) > \omega_C(\pi')
        \] due to \Cref{prop:CIstructurecone}. But \(C\) and \(C'\) satisfy the same inequalities, so in particular their
        critical paths satisfy the same inequalities. This means that \(\cone_\cg(C) = \cone_\cg(C')\) and since both weight
        matrices were chosen from the same cone of the Gr\"obner fan, \(\cone_\cg(C)\) is refined by the same cone
        and we get the statement.
    \end{proof}
\end{lemma}

The following example shows that the Gr\"obner fan of $J_\cg$ has a richer combinatorial structure than the maxoid fan and hence is often a strict refinement of the maxoid fan.

\begin{proposition}\label{prop:alternating-cycle-in-ideal}
    Let \(\cg\) be the following DAG \[
        \begin{tikzcd}
            1\arrow[d]\arrow[rd] & 2\arrow[d]\arrow[ld] \\
            4 & 3\arrow[l]
        \end{tikzcd}.
    \] Then, a universal Gr\"obner basis of \(J_\cg = \langle\, x_{14} - x_{13}x_{34}, x_{24} - x_{23}x_{34}\,\rangle\)
    contains the binomial \(x_{13}x_{24} - x_{14}x_{23}\).
    \begin{proof}
        Suppose \(<\) is a term ordering on \(k[E]\) such that \(\initial[<](x_{14} - x_{13}x_{34}) = x_{13}x_{34}\)
        and \(\initial[<](x_{24} - x_{23}x_{34}) = x_{23}x_{34}\). Then, the S-pair of those two polynomials is \[
          x_{23}(x_{14} - x_{13}x_{34}) - x_{13}(x_{24} - x_{23}x_{34}) = x_{13}x_{24} - x_{14}x_{23}. \qedhere
        \]
  \end{proof}
\end{proposition}

This means that compared to the maxoid fan the Gr\"obner fan of \(J_\cg\) contains a hyperplane \(c_{13}+c_{24} = c_{14} + c_{23}\)
subdividing \(\mathbb{R}^E\) further. But this hyperplane does not compare a pair of paths with common endpoints.
Thus, by \Cref{lem:markovcritpaths} being on either side of this hyperplane has no effect on the resulting maxoid.
This has the following consequence for the Gr\"obner fan of \(J_\cg\) with respect to the maxoids.

\begin{theorem}
    If the DAG from \Cref{prop:alternating-cycle-in-ideal} is an induced minor of \(\cg\), then the
    Gröbner fan of \(J_\cg\) is a strict refinement of the maxoid fan.

    \begin{proof}
        First, we restrict to the case where \(\cg\) has the induced subgraph \[
            \mathcal{H} \colon\begin{tikzcd}
                1\arrow[d]\arrow[rd] & 2\arrow[d]\arrow[ld] \\
                4 & 3\arrow[l]
            \end{tikzcd}
        \] and that \(\mathcal H\) is induced without loss of generality by \(\{1,2,3,4\}\).
        We know that \[{x_{14} - x_{13}x_{34},x_{24} - x_{23}x_{34}\in J_\cg}\] and by \Cref{prop:alternating-cycle-in-ideal}
        this also means that \(f \coloneqq x_{13}x_{24} - x_{14}x_{23}\in J_\cg\).

        Now, choose two generic weight matrices \(C',C''\in\mathbb{R}^E\) for \(\cg\) such that the the critical paths
        in \(\cg\) induced by \(C'\) and \(C''\) agree, but \(\initial(f) \neq \initial[C'](f)\). We can arrange for this
        by starting with such a generic matrix $C$ lying on the hyperplane $L$ defined by $c_{13}+c_{24} = c_{14} + c_{23}$ and setting $C' := C + \varepsilon u$ and $C'' := C - \varepsilon u$ , where $u$ is the normal vector to $L$ and $\varepsilon >0$ is small enough as to preserve the critical paths of $C$.
        By~\Cref{lem:markovcritpaths}, this means that \(\globalCsep{\cg}{C'} = \globalCsep{\cg}{C''}\),
        but \(C'\) and \(C''\) are in different full-dimensional cones of the Gr\"obner fan of \(J_\cg\). Note that $L$ is not a defining hyperplane of the maxoid fan, as it does not correspond to a pair of paths.

        If $\mathcal{H}$ is an induced minor of \(\cg\) with edge set \(E\), then there is a sequence of edge contractions
        that transforms an induced subgraph $\mathcal{H}'$ with edge set \(E'\) of \(\cg\) into $\mathcal{H}$.
        For example, if \(i\to k\to j\) is a path in $\mathcal{H}'$ and $\mathcal{H}$ is obtained by contraction of
        \(k \to j\), then \(I_{\mathcal{H}'}\) and \(I_\mathcal{H}\) differ only by the substitution of the variable \(x_{ij}\) with \(x'_{ik}x'_{kj}\)
        where \(x_{ij}\) and \(x'_{ij}\) denote the variables in \(k[E]\) respectively \(k[E']\).
    \end{proof}
\end{theorem}

\section{Conditional Independence Implication for Maxoids}
\label{sec:ci-axioms}

The polyhedral fan $\mathcal F_\cg$ provides the maxoids associated to a given DAG with a discrete geometric structure which is both interesting and practically useful. In this section we discuss an application to the \emph{implication problem} for conditional independence in the setting of maxoids.
A similar connection to polyhedral geometry has been previously exploited for CI~implication in the framework of \emph{structural imsets}~\cite{BHLS10}. However, the polyhedral fan in our case is specific to the graph and the map from its cones to maxoids does not in general induce a bijection, much less a Galois connection, as seen in \Cref{eg:3nodeDAG}. As a result, the extraction of conditional independence features from the polyhedral geometry is not straightforward.

\begin{definition}[CI~implication for maxoids] \label{def:CIimpl}
Let $\mathfrak m$ be a collection of maxoids together with CI~statements $p_1, \dots, p_s$ and $q_1, \dots, q_t$ over ground set~$V$. The \emph{conditional independence implication problem} for $\mathfrak m$ asks whether the implication
\[
  \bigwedge_{i=1}^s p_i \implies \bigvee_{j=1}^t q_j
\]
holds, i.e., whether every maxoid in $\mathfrak m$ containing all of the $p_i$ contains at least one of the~$q_j$. The following instances are of particular interest.
\begin{description}
\item[Local version:] The graph $\cg$ is fixed and $\mathfrak m$ is the set of all $\cg$-maxoids, i.e., maxoids of the form $\globalCsep{\cg}{C}$ for some weight matrix~$C$ compatible with~$\cg$.
\item[Local generic version:] The graph $\cg$ is fixed and $\mathfrak m$ consists of the generic $\cg$-maxoids, i.e., the $\globalCsep{\cg}{C}$ with $C \notin \mathcal{H}_\cg$.
\item[Global version:] The vertex set $V$ is fixed and $\mathfrak m$ is the collection of all $\cg$-maxoids over all DAGs~$\cg$ on~$V$.
\item[Global generic version:] The vertex set $V$ is fixed and $\mathfrak m$ is the collection of all maxoids which are generic for at least one DAG on vertex set~$V$.
\end{description}
\end{definition}

The global implication problem reduces to a series of local instances. To solve the local problem for a given graph~$\cg$, one could enumerate all $\cg$-maxoids via the cones of~$\cf_\cg$ and verify the implication case by case. However, there may be many more cones in this fan than there are maxoids satisfying the premises of the implication, so this method could be needlessly slow.
We propose a different approach which leverages the polyhedral geometry even more directly.

\begin{definition} \label{def:PolyCI}
For a given DAG~$\cg$ and a CI~statement $\CI{i,j|K}$ let $\PolyCI_\cg{i,j|K}$ be the set of all weight matrices~$C$ such that $\CI{i,j|K} \in \globalCsep{\cg}{C}$.
\end{definition}

\begin{lemma}
Let $\cg = (V, E)$ be fixed and let $i, j \in V$ distinct and $K \subseteq V \setminus ij$ be given. Then $\PolyCI_\cg{i,j|K}$ is a polyhedral set, i.e., a finite union of polyhedra.
\end{lemma}

\begin{proof}
Note that polyhedral sets are closed under finite unions, intersections and set difference. Thus it suffices to show that the set of all $C$ such that $\CI[\not\Cperp]{i,j|K}$ holds is~polyhedral. This set, in turn, is the union of five sets corresponding to the five types of $\ast$-connecting paths in $\critdag{C,K}$ pictured in \Cref{fig:*-sep}. The existence of a $\ast$-connecting path of a given type depends on the existence of edges in $\critdag{C,K}$. Now let $k, l \in V$, let $P$ be the set of all paths from~$k$ to~$l$ which contain a node from~$K$ in their interior, and let $P'$ be the other paths from~$k$ to~$l$. The edge $k \to l$ exists in $\critdag{C,K}$ if and only if $P\cup P' \neq \emptyset$ and $C$ satisfies the following polyhedral conditions:
\begin{equation}
  \label{eq:EdgeInCrit}
  \bigwedge_{\pi \in P} \bigvee_{\pi' \in P'} \omega_C(\pi') > \omega_C(\pi).
\end{equation}
These conditions express that no path from $k$ to $l$ in $\cg$ which intersects~$K$ can be critical, which precisely matches \Cref{def:*-sep}. The existence of a $\ast$-connecting path is a disjunction over conjunctions of formulas of type~\eqref{eq:EdgeInCrit} which is still polyhedral.
\end{proof}

\begin{remark}
The hyperplane arrangement $\mathcal{H}_\cg$ on which the non-generic weight matrices lie is a polyhedral set as well. Hence, the generic weight matrices for $\CI{i,j|K}$ with respect to $\cg$ form the polyhedral set $\PolyCI_\cg{i,j|K} \setminus \mathcal{H}_\cg$. Thus, all versions of the CI~implication problem highlighted in \Cref{def:CIimpl} reduce to a series of feasibility tests on polyhedral sets.
\end{remark}

\begin{example} \label{eg:Impl}
Fix the complete topologically ordered DAG $\cg$ on four vertices. The~poset of $\cg$-maxoids is pictured in \Cref{fig:maxoid-poset} and \Cref{fig:hasse-diagram-maxoids} lists the maxoids explicitly. Among~these nine maxoids precisely two satisfy $[1 \indep 4 \mid 3]$ (namely \tikz[baseline]{\definecolor{maxoid_1}{RGB}{51,34,136}\node[yshift=3pt, text=white, fill=maxoid_1, draw=maxoid_1, inner sep=2pt, rectangle, rounded corners=3pt] {$1$}} and \tikz[baseline]{\colorlet{maxoid_6}{orange!65}\node[yshift=3pt, text=black, fill=maxoid_6, draw=maxoid_6, inner sep=2pt, rectangle, rounded corners=3pt] {$6$}}) and both of them also satisfy $[2 \indep 4 \mid 1,3]$. Hence, the following is a valid local CI~implication for $\cg$:
\begin{equation}
  \label{eq:Impl}
  [1 \indep 4 \mid 3] \implies [2 \indep 4 \mid 1,3].
\end{equation}
This can be proved without first enumerating all $\cg$-maxoids. Associate to each CI~statement appearing in \eqref{eq:Impl} the polyhedral set according to \Cref{def:PolyCI}:
\begin{itemize}
\item For $[1 \indep 4 \mid 3]$, consider all types of $\ast$-connecting paths between~$1$ and~$4$ in~$\critdag{C,3}$. Since $1$ has no parents and $4$ has no children, there are only two possibilities. For~the edge $1 \to 4$, there are four paths in $\cg$ from~$1$ to~$4$ and we must ensure that neither $1 \to 3 \to 4$ nor $1 \to 2 \to 3 \to 4$ is critical:
\begin{equation}
  \label{eq:P1}
  \begin{gathered}
  \varphi =
  \Big[ c_{14} > c_{13} + c_{34} \;\lor\; c_{12} + c_{24} > c_{13} + c_{34} \Big] \land {} \\
  \Big[ c_{14} > c_{12} + c_{23} + c_{34} \;\lor\; c_{12} + c_{24} > c_{12} + c_{23} + c_{34} \Big].
  \end{gathered}
\end{equation}
The other possible $\ast$-connecting paths are $4 \ot 1 \to 3 \ot 1$ and $4 \ot 2 \to 3 \ot 1$ of type \Cref{fig:*-sep:d}. Their existence is characterized by
\begin{equation}
  \label{eq:P2}
  \varphi \lor c_{24} > c_{23} + c_{34}.
\end{equation}
The set $\PolyCI_\cg{1,4|3}$ is the complement of the union of the two sets described by \eqref{eq:P1} and \eqref{eq:P2}.
\item For $[2 \indep 4 \mid 1,3]$ the only possible $C^\ast$-connecting path is the direct edge $2 \to 4$ in $\critdag{C,13}$. This edge exists if and only if $c_{24} > c_{23} + c_{34}$.
\end{itemize}
The implication \eqref{eq:Impl} holds if and only if the polyhedral set $\PolyCI_\cg{1,4|3} \setminus \PolyCI_\cg{2,4|1,3}$ is empty. To solve this feasibility problem we apply state-of-the-art \emph{satisfiability modulo theories (SMT)} solvers with support for linear arithmetic over the ordered field of the rationals \cite{Satisfiability.jl,Z3}:
\begin{minted}{julia-repl}
julia> using Satisfiability;
julia> @satvariable(C[1:4, 1:4], Real);
julia> P143 = not(or(and(
         or(C[1,4] > C[1,3]+C[3,4], C[1,2]+C[2,4] > C[1,3]+C[3,4]),
         or(C[1,4] > C[1,2]+C[2,3]+C[3,4], C[1,2]+C[2,4] > C[1,2]+C[2,3]+C[3,4])
       ),   C[2,4] > C[2,3]+C[3,4]));
julia> P2413 = not(C[2,4] > C[2,3]+C[3,4]);
julia> counterexample = and(P143, not(P2413));
julia> sat!(counterexample)
:UNSAT
\end{minted}
The output of the \mintinline{julia}{sat!} call is \mintinline{julia}{:UNSAT} meaning that the queried formula is not satisfiable. This formula describes the matrices~$C$ for which $[1 \indep 4 \mid 3]$ holds but $[2 \indep 4 \mid 1,3]$ does not. From the non-existence of such a counterexample we can conclude that~\eqref{eq:Impl} is true for all $\cg$-maxoids. In fact, in this case the implication rewrites to $\neg(\varphi \lor \psi) \implies \neg\psi$, for $\varphi$ given by \eqref{eq:P1} and $\psi = c_{24} > c_{23} + c_{34}$ as in \eqref{eq:P2}. This is of course a logical tautology.
\end{example}

Our package \verb|Maxoids.jl| automates the derivation of the polyhedral sets~$\PolyCI_\cg{i,j|K}$ and includes a convenient interface for solving the local, generic and global CI~implication problems. In case a counterexample is determined to exist, the SMT solver also returns~it. For example, the reverse implication to~\eqref{eq:Impl} is not true. We can verify this and even obtain a certificate in the form of a weight matrix $C$ such that $(\cg, C)$ satisfies $[2 \indep 4 \mid 1,3]$ but not $[1 \indep 4 \mid 3]$:
\begin{minted}{julia-repl}
julia> using Maxoids, Oscar;
julia> G = complete_DAG(4);
julia> maxoid_implication(G, [ CI"24|13" ] => [ CI"14|3" ])
(false, [-Inf 1/4 0 1/2; -Inf -Inf 0 0; -Inf -Inf -Inf 0; -Inf -Inf -Inf -Inf])
\end{minted}
Note that the implication \eqref{eq:Impl} is specific to the maxoids realized by the complete graph~$\cg$. In a graph where $1$ is an isolated vertex, say, it need not hold. Finding such an example is a simple matter of applying the SMT solver in ``global mode'' where the returned certificate consists of a DAG $\cg$ supporting a counterexample as well as the weights:
\begin{minted}{julia-repl}
julia> maxoid_implication(4, [ CI"14|3" ] => [ CI"24|13" ])
(false, [2 => 4], [-Inf -Inf -Inf -Inf; -Inf -Inf -Inf 0; -Inf -Inf -Inf -Inf; …])
\end{minted}

We now consider properties that all maxoids, independently of the graphs they are realized by, have in common. Like many other types of graphical models (cf.~\cite{UnifyingMarkov}), maxoids satisfy the \emph{compositional graphoid properties}, i.e., every maxoid is closed under the following equivalence and implications for all disjoint $I,J,K,L \subseteq V$:
\begin{alignat}{3}
    \textbf{Semigraphoid:} &\quad \CI{I,J|L} \land \CI{I,K|JL} &&\iff\; \CI{I,JK|L}, \\
    \textbf{Intersection:} &\quad \CI{I,J|KL} \land \CI{I,K|JL} &&\;\implies\; \CI{I,JK|L}, \\
    \textbf{Composition:}  &\quad \CI{I,J|L} \land \CI{I,K|L} &&\;\implies\; \CI{I,JK|L}.
\end{alignat}
Whereas the Semigraphoid property holds for the CI~statements satisfied by any random vector, Intersection and Composition provide non-trivial additional structure~\cite{Compo}. For~example, Intersection guarantees the uniqueness of Markov boundaries \cite{PearlPaz} and Composition ensures the correctness of the IAMB algorithm to find them \cite{MarkovBoundary}.
Compositional graphoids also admit a significant complexity reduction in handling CI~statements, namely every compositional graphoid satisfies the equivalence
\[
  \CI{I,J|K} \iff \bigwedge_{i \in I} \bigwedge_{j \in J} \CI{i,j|K},
\]
for all disjoint $I, J, K \subseteq V$. Hence, the independence of $I$ and $J$ depends only on \emph{pairwise} interactions conditional on~$K$.

In~\cite{MaxLinearCI} it is mentioned without proof that $C^\ast$-separation satisfies the compositional graphoid properties. We~supply the routine proof below and then delve into the question of what distinguishes maxoids from other types of compositional graphoids.

\begin{proposition}
\label{prop:maxoids-are-comp-graphoids}
Maxoids are compositional graphoids.
\end{proposition}

\begin{proof}
Consider any maxoid $\mathcal{M} = \globalCsep{\cg}{C}$ for a given DAG $\cg$ and weight matrix~$C$ supported on~$\cg$. All~separation statements below are with respect to~$(\cg, C)$.
By~the definition of $C^\ast$-separation, the assumption $\CI[\not\Cperp]{I,J|L}$ implies the existence of a $\ast$-connecting path $\pi$ between $I$ and $J$ in the critical DAG~$\critdag{C,L}$. A~fortiori, $\pi$ also connects $I$ and $KL$ in $\critdag{C,L}$, hence $\CI[\not\Cperp]{I,JK|L}$.
Now consider a $\ast$-connecting path $\pi$ between $I$ and $K$ in $\critdag{C,JL}$. If it contains a collider $j \in J$, then the portion of $\pi$ from $I$ to $j$ is a $\ast$-connecting path between $I$ and $J$ in~$\critdag{C,L}$. Other\-wise the collider (if any) is in $L$ and $\pi$ yields a $\ast$-connecting between $I$ and $K$ in $\critdag{C,L}$. In both cases, we obtain a $\ast$-connecting path between $I$ and $JK$ in~$\critdag{C,L}$.
By~contra\-position, these two arguments prove the ``only if'' part of the Semigraphoid property.

The ``if'' direction is proved by contraposition as well. Assume that $\CI[\not\Cperp]{I,JK|L}$ and $\CI[\Cperp]{I,J|L}$, i.e., there exists a $\ast$-connecting path $\pi$ from $I$ to $JK$ but not one from $I$ to $J$ in~$\critdag{C,L}$. Hence, $\pi$ must connect $I$ and $K$ and cannot contain any node from~$J$. But then $\pi$ also $\ast$-connects $I$ and $K$ in~$\critdag{C,JL}$, thus $\CI[\not\Cperp]{I,K|JL}$ holds.

For Intersection, use again contraposition. Assume $\CI[\not\Cperp]{I,JK|L}$ and $\CI[\Cperp]{I,J|KL}$. By the Semigraphoid property and the symmetry with respect to exchanging $J$ and $K$, we can split $\CI[\not\Cperp]{I,JK|L}$ into two cases: $\CI[\not\Cperp]{I,K|L}$ or $\CI[\not\Cperp]{I,J|KL}$. The second case contradicts our other assumption.
In the former case, let $\pi$ denote a $\ast$-connecting path between $I$ and $K$ in~$\critdag{C,L}$. We may assume that this path is as short as possible, i.e., does not contain any other node of~$K$. If it contains a node $j\in J$, then the portion from $I$ to $j$ $\ast$-connects $I$ and $J$ in~$\critdag{C,KL}$ which is impossible. Hence $\pi$ is free of nodes from $J$ and thus $\ast$-connects $I$ and $K$ also in~$\critdag{C,JL}$ which is the required conclusion of Intersection.

The Composition property holds almost by definition. Any $\ast$-connecting path from $I$ to $JK$ in~$\critdag{C,L}$ connects either $I$ to $J$ or $I$ to $K$, which is the contrapositive of the assertion of Composition.
\end{proof}

Additionally, maxoids satisfy the following Amalgamation property for blocking sets:
\begin{alignat}{3}
    \textbf{Amalgamation:} &\quad \CI{i,j|KM} \land \CI{i,j|LM} &&\implies \CI{i,j|KLM},
\end{alignat}
for all distinct $i, j \in V$ and pairwise disjoint sets $K, L, M \subseteq V \setminus ij$. This property reflects the graphical definition of ``blocked critical paths'' that underlies $C^\ast$-separation. Notably,~its proof relies on the fact that $\ast$-connecting paths can have at most one collider. In~the Bayesian network for the ``Cassiopeia graph'' $1 \to 2 \ot 3 \to 4 \ot 5$ we have $\CI[\dperp]{1,5|2}$ and $\CI[\dperp]{1,5|4}$ but $\CI[\not\dperp]{1,5|2,4}$, so Amalgamation fails.

\begin{proposition} \label{prop:Amalgam}
Maxoids satisfy the Amalgamation property.
\end{proposition}

We first establish a simple monotonicity property for edge sets in critical~DAGs.

\begin{lemma} \label{lemma:CritMonotone}
Let $(\cg, C)$ be a weighted DAG. For $M' \subseteq M$ we have $E(\critdag{C,M}) \subseteq E(\critdag{C,M'})$.
\end{lemma}

\begin{proof}
Let $i \to j$ be an edge in $\critdag{C,M}$. Recall from \Cref{def:*-sep} that there is a path from $i$ to $j$ in $\cg$ and no critical path between $i$ and $j$ in $\cg$ intersects~$M$. But then no critical path intersects the subset~$M'$ and so $i \to j$ is an edge in $\critdag{C,M'}$ as well.
\end{proof}

\begin{proof}[Proof of \Cref{prop:Amalgam}]
Fix any $(\cg, C)$ and consider its maxoid. Arguing by contraposition, we assume $\CI[\not\Cperp]{i,j|KLM}$ and take a $\ast$-connecting path $\pi$. By \Cref{lemma:CritMonotone} all edges on $\pi$ also exist in both critical DAGs~$\critdag{C,KM}$ and~$\critdag{C,LM}$. If $\pi$ involves parents $p$ and/or $q$ (cf.~\Cref{fig:*-sep}) then these nodes are outside of $KLM$ and hence outside of~$KM$ and of~$LM$. If~$\pi$ involves a collider $\ell$ then it must be inside $KLM$ and hence in~$KM$ or in~$LM$. Thus,~$\pi$ also $\ast$-connects $i$ and $j$ in $\critdag{C,KM}$ or in $\critdag{C,LM}$ which was the claim.
\end{proof}

Spohn in \cite[Eq.~(S5)]{SpohnProperties} found an implication which is valid for random vectors with strictly positive density. Maxoids satisfy a stronger version of this implication.
\begin{alignat}{3}
    \label{eq:Spohn1}
    \textbf{Strong Spohn:} &\; \CI{i,j|klM} \land \CI{k,l|iM} \land \CI{k,l|jM} &&\implies \CI{k,l|M}, \\ %
    \label{eq:Spohn2}
                    &\quad \CI{i,j|klM} \land \CI{k,l|iM} \land \CI{k,l|M} &&\implies \CI{k,l|jM},
\end{alignat}
for all distinct $i, j, k, l \in V$ and $M \subseteq V \setminus ijkl$.

\begin{remark}
The original Spohn property also includes $\CI{k,l|ijM}$ in the premises of both implications. For implication~\eqref{eq:Spohn1} this additional premise is redundant in the maxoid setting by virtue of \Cref{prop:Amalgam}. By contrast, no strict subset of the premises of~\eqref{eq:Spohn2} implies $\CI{k,l|ijM}$. Counterexamples can be found using an SMT solver:
\begin{minted}{julia-repl}
julia> maxoid_implication(4, [ CI"14|23", CI"23|1" ] => [ CI"23|4" ])
(false, [1 => 2, 1 => 3], [-Inf 0 0 -Inf; -Inf -Inf -Inf -Inf; -Inf -Inf -Inf -Inf; …])
julia> maxoid_implication(4, [ CI"14|23", CI"23|" ] => [ CI"23|4" ])
(false, [2 => 4, 3 => 4], [-Inf -Inf -Inf -Inf; -Inf -Inf -Inf 0; -Inf -Inf -Inf 0; …])
julia> maxoid_implication(4, [ CI"23|", CI"23|1" ] => [ CI"23|4" ])
(false, [2 => 4, 3 => 4], [-Inf -Inf -Inf -Inf; -Inf -Inf -Inf 0; -Inf -Inf -Inf 0; …])
\end{minted}
Hence, it is remarkable that \eqref{eq:Spohn2} holds. The Strong Spohn property, just like the original one, has an overlap in the premises between the two implications. To highlight this, the Strong Spohn property can be restated as: given that $\CI{i,j|klM}$ and $\CI{k,l|iM}$ hold, the two statements $\CI{k,l|M}$ and $\CI{k,l|jM}$ are equivalent; see also \cite[Lemma~24]{SimecekDiss}.
\end{remark}

\begin{lemma} \label{lemma:SplitCritical}
Let $(\cg, C)$ be a weighted DAG and suppose that $i \to k \in E(\critdag{C,M})$. If $\pi$ is a critical path from $i$ to $k$ in $(\cg, C)$ containing a node $j$ distinct from $i$ and $k$, then $i \to j$ and $j \to k$ also exist in~$\critdag{C,M}$.
\end{lemma}

\begin{proof}
Split $\pi$ into two portions: $\pi_0$ from $i$ to $j$ and $\pi_1$ from $j$ to $k$. For any path $\pi_0'$ from $i$ to $j$ in $\cg$, we can form the path $\pi'$ from $i$ to $k$ by replacing $\pi_0$ in $\pi$ by $\pi_0'$. Since $\pi$ is a critical path, we have $\omega_C(\pi_0) + \omega_C(\pi_1) = \omega_C(\pi) \ge \omega_C(\pi') = \omega_C(\pi_0') + \omega_C(\pi_1)$, and it follows that $\omega_C(\pi_0) \ge \omega_C(\pi_0')$. Thus $\pi_0$ is critical. If $\pi_0'$ is also critical, then $\pi'$ is a critical path from $i$ to $k$. By assumption $\pi'$ and hence $\pi_0'$ do not contain any node from~$M$. This shows $i \to j \in E(\critdag{C,M})$. The same argument works when $\pi_1$ is replaced by any critical path $\pi_1'$ between $j$ and $k$.
\end{proof}

\begin{proposition}
Maxoids satisfy the Strong Spohn property.
\end{proposition}

\begin{proof}
Fix $(\cg, C)$. For the first implication we prove the equivalent statement
\[
  \CI[\not\Cperp]{k,l|M} \land \CI[\Cperp]{k,l|iM} \land \CI[\Cperp]{k,l|jM} \implies \CI[\not\Cperp]{i,j|klM}.
\]
Let $\pi$ be a $\ast$-connecting path between $k$ and $l$ in $\critdag{C,M}$. Thus $\pi$ consists of up to four edges among up to five vertices. Each of these edges $x \to y$ witnesses that in $\cg$ there is a directed path from~$x$ to~$y$ and that no critical path between them intersects~$M$. We may additionally assume that none of these critical paths contain $k$ or $l$ as interior vertices. Otherwise \Cref{lemma:SplitCritical} guarantees that we can choose a shorter $\ast$-connecting path $\pi$ between $k$ and $l$ in $\critdag{C,M}$.

Since $\CI[\Cperp]{k,l|iM}$ holds, $\pi$ cannot be $\ast$-connecting in $\critdag{C,iM}$. If $\pi$ contains a collider $m \in M$, then this collider is also in~$iM$, so the reason why $\pi$ no longer connects $k$ and $l$ in $\critdag{C,iM}$ must be that one of the edges in $\pi$ no longer exists in $\critdag{C,iM}$ because a corresponding critical path in $\cg$ contains~$i$. (This~includes the special cases when $\pi$ includes parents $p$ and/or $q$ and $p = i$ or $q = i$; cf.~\Cref{fig:*-sep}.) Similar reasoning applies to~$\critdag{C,jM}$.
We~distinguish two cases:
\begin{paraenum}[label=(\alph*)]
\item Suppose that the same edge, say $x \to y$, of $\pi$ is missing in both $\critdag{C,iM}$ and $\critdag{C,jM}$. Note that $y \in klM$. We have already argued in the beginning that $k$ and $l$ do not appear on any critical path from~$x$ to~$y$ in~$\cg$. Thus, $x \to y$ is an edge in $\critdag{C,klM}$. Since this edge is missing in $\critdag{C,iM}$ and $\critdag{C,jM}$ there exist a critical path from $x$ to $y$ in $\cg$ containing $i$ and a critical path from $x$ to $y$ containing $j$, both containing no internal vertex from~$klM$. \Cref{lemma:SplitCritical} shows that also $i \to y \ot j$ exists in $\critdag{C,klM}$ and hence $\CI[\not\Cperp]{i,j|klM}$.

\item Otherwise $\pi$ is missing one edge in $\critdag{C,iM}$ and another in $\critdag{C,jM}$. Thus, there are critical paths corresponding to these edges which contain $i$ and $j$, respectively. These paths do not contain $k$ or $l$ and thus \Cref{lemma:SplitCritical} ensures that a $\ast$-connecting path between~$i$ and~$j$ exists in $\critdag{C,klM}$ which proves $\CI[\not\Cperp]{i,j|klM}$.
\end{paraenum}

The second implication is proved similarly, again in rearranged form:
\[
  \CI[\not\Cperp]{k,l|jM} \land \CI[\Cperp]{k,l|M} \land \CI[\Cperp]{k,l|iM} \implies \CI[\not\Cperp]{i,j|klM}.
\]
Let $\pi$ be a $\ast$-connecting path between $k$ and $l$ in $\critdag{C,jM}$. As before we may assume that no critical path associated to any of the edges in~$\pi$ contains $k$ or $l$, as otherwise we may shorten~$\pi$. By \Cref{lemma:CritMonotone} the edges on $\pi$ also exist in $\critdag{C,M}$ but the separation assumptions imply that $\pi$ does not $\ast$-connect. This is only possible if $\pi$ contains $j \notin M$ as a collider. The separation $\CI[\Cperp]{k,l|iM}$ implies that $\pi$ does not exist in $\critdag{C,iM}$. Hence, $i$ appears on some critical path for an edge in~$\pi$ (or equals one of the parents $p$ and/or $q$). As before, \Cref{lemma:SplitCritical} then gives a $\ast$-connecting subpath between~$i$ and~$j$ in~$\critdag{C,klM}$.
\end{proof}

We end this section with a discussion of alternate representations of maxoids. Outside of graphical models, abstract properties of conditional independence are well-studied only for discrete and regular Gaussian random variables; see~\cite{StudenyIngleton,Dissert}. The parametrization of MLBNs in \eqref{eqn:mlbn} does not produce jointly Gaussian distributions as the maximum of Gaussians does not follow a Gaussian distribution. On the other hand, discrete distributions are not atom-free and are thus incompatible with this parametrization. Nevertheless, it is reasonable to ask whether maxoids, as abstract conditional independence models, can be represented using one of these two distribution classes.
For Gaussian distributions we can present an easy obstruction. Drton and Xiao~\cite{DrtonXiao} coined the term \emph{semigaussoid} as a synonym for ``compositional graphoid''. What is missing from a semigaussoid to a \emph{gaussoid} is the closedness under %
\begin{alignat}{3}
    \textbf{Weak Transitivity:} &\quad \CI{i,j|L} \land \CI{i,j|kL} &&\iff \CI{i,k|L} \lor \CI{j,k|L},
\end{alignat}
for all distinct $i, j, k$ and $L \subseteq V \setminus ijk$. The following example shows that maxoids need not satisfy Weak Transitivity. By results of \cite{LnenickaMatus}, this provides a maxoid which cannot be faithfully represented by a regular Gaussian random vector. This phenomenon sets maxoids further apart from d-separation graphoids.

\begin{example}
Consider the diamond graph $\cg$ as described in \Cref{eg:diamond} with weight matrix $C$ satisfying $\omega_C(\pi_3) > \omega_C(\pi_2)$.
The maxoid of $(\cg, C)$ consists precisely of the d-separations in $\cg$ plus $\CI{1,4|3}$. %
As $\CI{1,4|3}$ and $\CI{1,4|2,3}$ hold without $\CI{1,2|3}$ or $\CI{2,4|3}$, this CI~structure violates Weak Transitivity and cannot be faithfully represented by a regular Gaussian distribution.
This violation of Weak Transitivity is reflected in the model geometry as follows. The space of regular Gaussian distributions which are Markov to this CI~structure is the union of two standard Bayesian networks on subgraphs of the diamond $\cg$: one has the edge $1 \to 2$ removed (so that it satisfies $\CI{1,2|3}$) and the other has the edge $2 \to 4$ removed (and hence satisfies $\CI{2,4|3}$). For an algebraic explanation of this phenomenon we refer to~\cite{Faithless}.
\end{example}

We were not able to find any apparent reasons why maxoids cannot be representable by discrete random vectors and pose an even stronger conjecture.

\begin{conjecture} \label{conj:Discrete}
Maxoids are representable by positive discrete distributions.
\end{conjecture}

This conjecture is reminiscent of the open problem of Studený \cite[Question~3, p.~191]{Studeny} concerning the representability of regular Gaussian CI~structures by discrete (or even positive binary) distributions.
In support of this conjecture we proved that maxoids satisfy basic CI~properties of positive distributions: Intersection and the Spohn property. We have also verified that all maxoids on four nodes are indeed representable by (not necessarily positive) discrete distributions using the complete axiomatization provided by Studený in (E:1)--(E:5) and (I:1)--(I:19) of \cite[Section~V]{StudenyIngleton}.

\begin{remark}
An earlier version of this manuscript contained a claimed counterexample to \Cref{conj:Discrete}. However, the underlying computation was performed in two stages which used mutually inconsistent labels for the involved nodes. This led to a false~positive.
Our software package supports further experimentation on this conjecture. The source code of our computations in support of \Cref{conj:Discrete} is available in our data~repository.
\end{remark}

\section*{Acknowledgements}

\setlength{\intextsep}{5pt}%
\setlength{\columnsep}{5pt}%
\begin{wrapfigure}{R}{0.13\linewidth}
\vspace{-.5\baselineskip}%
\centering%
\href{https://doi.org/10.3030/101110545}{%
\includegraphics[width=0.9\linewidth]{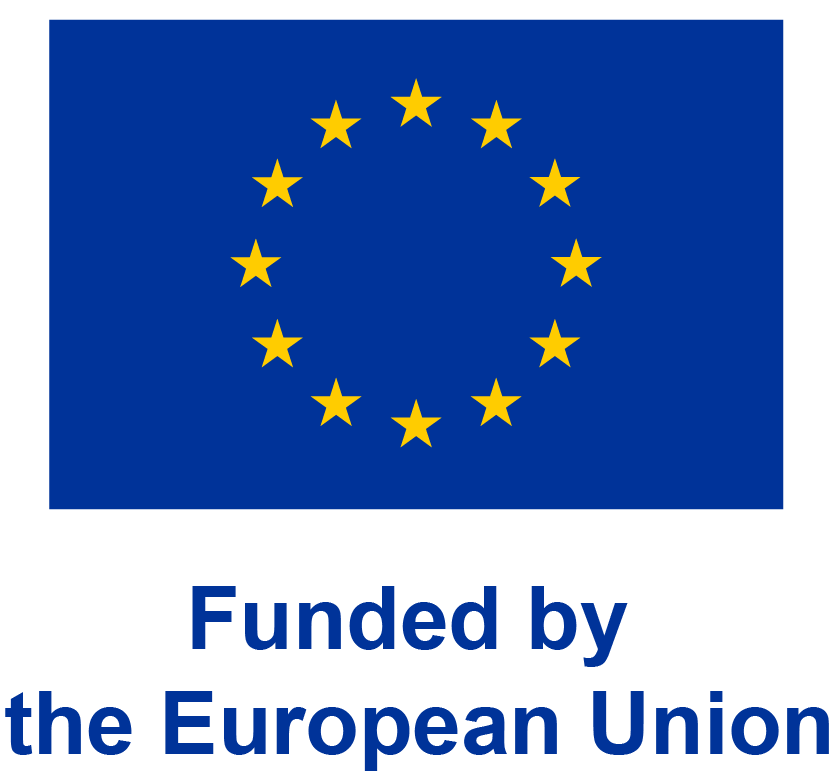}%
}
\end{wrapfigure}
T.B.~was funded by the European Union's Horizon 2020 research and innovation
programme under the Marie Skłodowska-Curie grant agreement No.~101110545.
K.F.~was funded by the Deutsche Forschungsgemeinschaft (DFG, German Research Foundation) under Germany's Excellence Strategy --- The Berlin Mathematics Research Center MATH+ (EXC-2046/1, project ID: 390685689).
B.H.~was supported by the Alexander von Humboldt Foundation.
F.N.~was funded by the Deutsche Forschungsgemeinschaft (DFG, German Research Foundation)
under the Priority Programme “Combinatorial Synergies” (SPP 2458, project ID: 539875257).

\bibliographystyle{plain}
\bibliography{maxlin}

\begin{appendices}
\begin{section}{Polyhedral and Tropical Geometry}

In this section, we provide an overview of the notions from polyhedral and tropical geometry which appear in this text. Especially in the parts on tropical geometry, we sacrifice complete mathematical rigor for the sake of readability. For comprehensive introductions of the notions used in this paper, we refer to Chapters 1 and 7 of \cite{Ziegler} and Chapters 1 and 2 of \cite{JoswigBook} respectively.
\subsection{Cones and Fans} \label{app:conesandfans}
Let $u \in \mathbb{R}^n \setminus \{ 0 \}$.
 The \emph{hyperplane} in $\mathbb{R}^n$ with normal vector $u$ is
    \begin{equation*}
        H_{u} := \bigl\{ x \in \mathbb{R}^n \ : \ \langle x, u \rangle = 0 \ \bigr\}.
    \end{equation*}
The ($n$-dimensional, closed, positive) \emph{half-space} defined by $u$ is
        \begin{equation*}
        H^{+}_{u} := \bigl\{ x \in \mathbb{R}^n \ : \ \langle x, u \rangle \geq 0 \ \bigr\}.
    \end{equation*}
    We call a set $C \subset \mathbb{R}^n$ a \emph{polyhedral cone} if there exists a finite collection of half-spaces $H_{u_1}$, ..., $H_{u_k}$ such that
\begin{equation} \label{hdescription}
    C = H^+_{u_1} \cap H^+_{u_1} \cap H^+_{u_2} \dots \cap H^+_{u_k}.
\end{equation}
     The dimension of $C$ is the dimension of the smallest vector space in which is it is fully contained.
     A \emph{supporting hyperplane} to $C$ is a hyperplane $H_u$ with $u \in \mathbb{R}^n \setminus \{ 0 \}$ such that
    \begin{center}
        $C = C \cap H^{+}_u \ \ \ \ \ $ (or equivalently, $\langle u, x \rangle \geq 0 $ for all $x \in C$).
    \end{center}
     A (proper) \emph{face} of $C$ is a set $F$ of the form
    \begin{center}
        $F = C \cap H_{u} \ \ \ \ \ $ , where $H_u$ is a supporting hyperplane of $C$.
    \end{center}
     The \emph{dimension} of a face $F$ is the dimension of $F$ as a cone.
     Faces of $C$ of dimension $\dim(C) -1$ are called \emph{facets}.
A polyhedral \emph{fan} in $\mathbb{R}^n$ is a finite collection $\mathcal{F}$ of non-empty polyhedral cones with the following two properties:
\begin{enumerate}[label = \textit{(\roman*)}]
    \item Every non-empty face of a cone in $\mathcal{F}$ is in $\mathcal{F}$.
    \item The intersection of any two cones is a face of both.
\end{enumerate}
    We say that $\mathcal{F}$ is \emph{complete} if
\begin{equation*}
    \mathbb{R}^n = \bigcup_{C \in \mathcal{F}} C,
\end{equation*}
whereas $\mathcal{F}$ is said to be \emph{pure} if all of its maximal cones are of the same dimension.
Complete fans are examples of polyhedral subdivisions of $\mathbb{R}^n$. A subclass of complete fans are those generated by \emph{hyperplane arrangements}, i.e. finite collections of linear hyperplanes. The combinatorics of the fan are determined by the intersecting of the hyperplanes. However, not every complete fan arises in this way.

\begin{subsection}{Polytopes and normal fans} \label{app:polytopes}
The \emph{affine hyperplane} in $\mathbb{R}^n$ with normal vector $u$ and displacement $d \geq 0$ is
    \begin{equation*}
        H_{u,d} := \bigl\{ x \in \mathbb{R}^n \ : \ \langle x, u \rangle = d \ \bigr\}.
    \end{equation*}
The corresponding \emph{(non-negative) affine half-space is}
\begin{equation*}
    H^+_{u,d} := \bigl\{ x \in \mathbb{R}^n \ : \ \langle x, u \rangle \geq d \ \bigr\}.
\end{equation*}
A \emph{polytope} $P$ in $\mathbb{R}^n$ is a bounded intersection of affine half-spaces, i.e. a set of the form
\begin{equation*}
     P = H^+_{u_1, d_1} \cap H^+_{u_2,d_2} \cap H^+_{u_3,d_3} \dots \cap H^+_{u_k,d_k}.
\end{equation*}
Analogously to the case of cones, a face of a $P$ is a set of the form $P \cap H_{u,d} $, where $H_{u,d}$ is a supporting affine hyperplane. Faces of polytopes are once again polytopes. Faces of dimensions 1, 2 and $\dim P -1$ are called \emph{vertices}, \emph{edges} , and \emph{facets} respectively. Some readers may be more familiar with polytopes as convex hulls of finitely many points: this notion and the half-space description presented here are equivalent (see Theorem 2.15 of \cite[p.~65]{Ziegler}).

Every polytope $P \subset \mathbb{R}^n$ is associated to a polyhedral fan in the dual space called the \emph{normal fan}, which is constructed as follows. For each non-empty face $F$ of $P$ consider the set of linear functionals $c \in (\mathbb{R}^n)^\ast$ which attain their maximal value in $P$ along $F$. That is,
\begin{equation*}
    N_F := \{ c \in (\mathbb{R}^n)^\ast: F \subset \{x\in P: \langle c, x \rangle = \max \langle c, y \rangle : y \in  P\} \}.
\end{equation*}
These sets are polyhedral cones in the sense of the previous section, and taken over all faces they form a fan in $\mathbb{R}^n$ which we denote $\mathcal{N}(P)$. The fan is complete, and if the polytope is full-dimensional then its normal fan is pure. There is an inclusion-reversing bijection between the $k$-dimensional faces of $P$ and the $(n-k)$-dimensional faces of $\mathcal{N}(P)$. In particular, each vertex of $P$ corresponds to a full-dimensional cone in $\mathcal{N}(P)$.

\subsection{Polynomials and Tropical Hypersurfaces} \label{app:trop}
Let $f \in \mathbb{R}[x_1,\dots, x_n]$ be a polynomial in $n$ variables with coefficients in a field $\mathbb{R}$. The \emph{ (max-plus) tropicalization} $Trop(f)$ of $f$ is the function obtained by replacing regular addition and multiplication with with the tropical operations $ \oplus := \max$ and $\odot := +$ respectively.
\begin{example} \label{eg:troppoly}
Let $f = xy + zw \in \mathbb{R}[x,y,z,w]$. Its tropicalization is the tropical polynomial
\begin{equation} \label{eqn:tropf}
    Trop(f) := x \odot y \oplus z \odot w = \max \{x+y, z+ w \}
\end{equation}

\end{example}
Seen as a function $\mathbb{R} \to \mathbb{R} $, $Trop(f)$ is a piecewise linear function. The \emph{tropical hypersurface} $\mathcal{T}(Trop(f))$ is the set of points where the maximum is attained by more than one term of $Trop(f)$. For the function $f$ defined as in \Cref{eg:troppoly}, its tropical hypersurface is the hyperplane $\{ x + y = z + w \}$. Every tropical hypersurface induces a \emph{polyhedral subdivision} of $\mathbb{R}^n$, the co-dimension 1 skeleton of which is the hypersurface itself. The full-dimensional cells of the subdivision correspond to the regions of $\mathbb{R}^n$ on which the value of $Trop(f)$ is attained by a unique term. \\
The \emph{Newton polytope} of a polynomial $f$ is the convex hull of the integer vectors corresponding to the exponents of its terms.
\begin{example} \label{eg:newtonpolytope}
    The Newton polytope of the polynomial $f$ from \cref{eg:troppoly} is \begin{equation}
        \Newt(f) := conv( \begin{pmatrix}
            1 \\ 1 \\ 0 \\ 0
        \end{pmatrix} , \begin{pmatrix}
            0 \\ 0 \\ 1 \\ 1 \end{pmatrix} )
    \end{equation}
    It is a 1-dimensional polytope in $\mathbb{R}^4$ (i.e. a line segment).
\end{example}
The tropical hypersurface induced by a polynomial $f$ with constant coefficients can be described in terms of $\Newt(f)$:
\begin{theorem}[{\cite[Corollary 1.21]{JoswigBook}}]\label{thm:jos}
Let $F = Trop(f)$ be a max-tropical polynomial with constant coefficients. Then the polyhedral subdivision induced by tropical hypersurface $\mathcal{T}(F)$ is precisely the outer normal fan of $\Newt(f)$.
\end{theorem}
\begin{example}
The normal fan of the polytope from \Cref{eg:newtonpolytope} is the fan in $\mathbb{R}^4$ determined by the plane $H = \{x + y = z +w\}$.
\end{example}
To study the polyhedral subdivision of $\mathbb{R}^n$ induced by a finite collection of tropical polynomials, we recall the notion of a \emph{Minkowski sum}: for two polytopes $P, Q \subset \mathbb{R}^n$, this is the set
\begin{equation} \label{eqn:minkowskisum}
    P + Q := \{ x + y \ , x \in P \ , y \in Q \} \ ,
\end{equation}
and is once again a polytope in $\mathbb{R}^n$. The Minkowski sum of finitely many polytopes is defined analogously. The operation of passing to the normal fan behaves nicely under Minkowski sums:

\begin{theorem}[{\cite[Prop.~7.12]{Ziegler}}]\label{thm:normalminkowskisum}
The normal fan of a Minkowski sum is the common refinement of the individual normal fans.
\end{theorem}
By \emph{common refinement} of two fans $\cf_1$ and $\cf_2$ we mean the collection of cones
\begin{equation*}
    \{ C_1 \cap C_2 , \ C_1 \in \cf_1 \ , \ C_2 \in \cf_2 \} \ ,
\end{equation*}
which is again a fan.

Combined with \Cref{thm:jos}, this fact gives the following characterization of the polyhedral subdivision induced by intersections of tropical hypersurfaces.

\begin{theorem} \label{thm:tropvarietyfan}
Let $f_1, \dots f_s \in \mathbb{R}[x_1 , \dots x_n]$ be polynomials with constant coefficients. The polyhedral subdivision of $\mathbb{R}^n$ induced by the tropical polynomials $F_i : = Trop(f_i)$ is the normal fan of the Minkowski sum of $\Newt(F_i)$, $i \in \{1, \dots s \}$.

\end{theorem}
\begin{proof}
    This follows directly from combining \Cref{thm:jos} and \Cref{thm:normalminkowskisum}.
\end{proof}
\end{subsection}
\end{section}

\begin{section}{Gr\"obner bases and fans}\label{app:GB}
In this section, we provide a short review of the terminology of Gr\"obner bases and Gr\"obner fans.

Fix a vector \(w\in\mathbb{R}^n_{\geq 0}\). A \emph{partial weight ordering} $<_w$ on \(R = \mathbb{R}[x_1,\dots,x_n]\)
is a relation on the monomials of \(R\) defined by
\[ x^\alpha <_w x^\beta : \iff \langle \alpha , w \rangle < \langle \beta , w \rangle. \]
This is not a total ordering in general but if it is, we say that \(<_w\) is a \emph{term ordering} on~\(R\).

The \emph{initial form} \(\mathrm{in}_w(f)\) of \(f\) with respect to \(w\) is defined as the sum of
the terms of \(f\) which are maximal with respect to $<_w$. The \emph{initial ideal} of a non-zero ideal \(I \trianglelefteq R\)
with respect to \(<_w\) is \[
  \mathrm{in}_w(I) = \langle\, \mathrm{in}_w(f) \mid f\in I \,\rangle.
\] If \(<_w\) is a term ordering, then \(\mathrm{in}_w(f)\) is always a monomial and also called the \emph{initial term}.
In this case, \(\mathrm{in}_w(I)\) is a monomial ideal.

A finite set $G = \{ g_1, \dots g_s\} \subset R$ of polynomials is called a \emph{Gröbner basis}
for an ideal $I$ with respect to the term ordering $<_w$ if \(I=\langle G\rangle\) and
\begin{equation*}
     \langle\, \mathrm{in}_w(g_1) \dots \mathrm{in}_w(g_s) \,\rangle = \mathrm{in}_w(I).
\end{equation*}

The equivalence classes of weight vectors with respect to the initial ideal \(\mathrm{in}_w(I)\) form a polyhedral fan
called the \emph{Gr\"obner fan}. This has been introduced by Mora and Robbiano~\cite{Mora.Robbiano:1998}.
In fact, the weight vectors \(w\in\mathbb{R}^n\) for which \(<_w\) is a term ordering lie in the interiors of full-dimensional cones,
corresponding to each monomial ideal \(\mathrm{in}_w(I)\).

In general, different choices of term orderings \(<w\) yield different Gr\"obner bases for a fixed ideal \(I\).
A finite subset \(G\subset I\) which is a Gr\"obner basis with respect to every term ordering \(<_w\) on \(\mathbb{R}[x_1,\dots,x_n]\)
is called \emph{universal}. Such a Gr\"obner basis is particularly useful for computing the Gr\"obner fan of an ideal.
\end{section}
\end{appendices}

\makecontacts

\end{document}